\documentclass[a4paper]{article}
\usepackage{graphicx} 

\usepackage{amssymb,amsthm,hyperref}
\usepackage{graphicx}
\usepackage{mathrsfs}
\usepackage{bbm}
\usepackage{amsfonts}
\usepackage{amsmath}
\usepackage{amscd}
\usepackage{float}
\usepackage{chngcntr}
\usepackage{apptools}
\usepackage{leftidx}
\usepackage[ruled,vlined]{algorithm2e}
\usepackage{indentfirst,csquotes}

\usepackage{bm}
\usepackage{multirow,multicol}
\usepackage{pgfplots}
\pgfplotsset{compat=1.18}
\usepackage{tikz}
\usetikzlibrary{arrows.meta}
\usetikzlibrary {shapes.geometric,patterns,hobby}

\newtheorem{remark}{Remark}
\newtheorem{Theorem}{Theorem}[section]

\newtheorem{Lemma}[Theorem]{Lemma}

\newtheorem{Proposition}[Theorem]{Proposition}
\newtheorem{Remark}[Theorem]{Remark}

\newcommand{\dx}{\,\mathrm{d}x}

\renewcommand{\vec}[1]{\mbox{\boldmath$#1$}}

\newcommand{\vw}{\vec{w}}
\newcommand{\vvv}{\vec{v}}

\newcommand{\vu}{\vec{u}}
\newcommand{\vf}{\vec{f}}

\renewcommand{\vec}[1]{\mbox{\boldmath$#1$}}

\renewcommand{\vec}[1]{\mbox{\boldmath$#1$}}

\usepackage{geometry}

\hypersetup{ colorlinks=true, linkcolor=black, filecolor=black, urlcolor=black }
\begin{document}
\title{Energy stable gradient flow schemes for shape and topology optimization in Navier-Stokes flows\thanks{This work was supported in part by the National Key Basic Research Program under grant 2022YFA1004402, the Science and Technology Commission of
Shanghai Municipality (No. 22ZR1421900 and 22DZ2229014), and the National Natural Science Foundation of China under grant (No. 12071149).
}}

\author{Jiajie Li \thanks{School of Mathematical Sciences, Shanghai Jiao Tong University, Shanghai 200240, China. E-mail: lijiajie20120233@163.com}
\and Shengfeng Zhu \thanks{Key Laboratory of MEA (Ministry of Education) \& Shanghai Key Laboratory of Pure Mathematics and Mathematical Practice \& School of Mathematical Sciences, East China Normal University, Shanghai 200241, China. E-mail: sfzhu@math.ecnu.edu.cn}}

\maketitle

\begin{abstract}
We study topology optimization governed by the incompressible Navier-Stokes flows using a phase field model. Novel stabilized semi-implicit schemes for the gradient flows of Allen-Cahn and Cahn-Hilliard types are proposed for solving the resulting optimal control problem. Unconditional energy stability is shown for the gradient flow schemes in continuous and discrete spaces. Numerical experiments of computational fluid dynamics in 2d and 3d show the effectiveness and robustness of the optimization algorithms proposed.
\end{abstract}




\section{Introduction}
Shape and topology optimization \cite{BS} of computational fluid dynamics is a popular topic with applications such as auto-coronary bypass anastomoses in medical science \cite{QDMV}, the laminar flow wing design in aeronautics \cite{MP} and pipe flow \cite{DFO}. Such optimal control problems in fluid flow \cite{MP} aim to seek a configuration or layout for optimizing some objective (e.g., energy dissipation and geometric inverse problem) subject to geometric and physical constraints such as incompressible fluid flows \cite{PF,PFS,LiCMAME,DLWW2013,GSH2005,LiNMPDE2023,SSchulz2010} or compressible Navier-Stokes equations \cite{PS}. Compared to shape optimization by adjusting the profile of geometric boundary to obtain better configuration \cite{GLZ2022}, topology optimization can perform both shape and topological changes of structures.  Numerical realization of topology optimization in fluid flows can be performed via the variable density method \cite{Borrvall2003}, topological derivative \cite{Guillaume}, level set method \cite{ZL}, phase field method \cite{LiFu, GarckeCOCV}, etc.

As diffusive interface tracking techniques, the phase field method \cite{GarckeCOCV,AMW,GHHK,Garcke2016,JLXZ2024phaseField} is introduced for minimizing general volume and surface functionals constrained by incompressible Navier-Stokes flows. The porous medium approach \cite{Borrvall2003} proposed by Borrvall and Petersson enables the governing equations to be defined on a fixed domain with a variable linear term characterizing the permeability. The main idea of the phase field model of topology optimization is to combine both objective and so-called Ginzburg–Landau energy to construct the total free energy. The latter is a diffuse interface approximation of perimeter regularization implying the existence of the optimal control problem \cite{GHHK}. For topology optimization constrained by Navier-Stokes flow, Garcke et al. \cite{Garcke2016,GarckeCOCV} discussed the differentiability of the solution to the phase field function, the first-order necessary condition from sensitivity analysis, and sharp interfacial asymptotic analysis.

The gradient flow method actually was a powerful tool for minimization of nonlinear or multi-physical coupling type of energy functional for many problems (see, e.g., \cite{Shen2019,DF}). From the numerical perspective, a gradient flow scheme is generally evaluated on the energy dissipation and its computational efficiency \cite{Shen2019}. The energy dissipative gradient flow scheme is of crucial importance to topology optimization implying that the sequence generated by the algorithm is convergent and monotonous. Unlike phase field models in the physical background such as interface dynamics \cite{AMW} or crystallization \cite{Elder2002}, the difficulty of constructing an energy dissipative gradient flow scheme for topology optimization arises from the coupling among the gradient flow, linear/nonlinear physical constraints of partial differential equations and possible extra adjoint systems. For shape design constrained by Stokes flow in the phase field model, an energy monotonic-decaying gradient flow scheme is proposed \cite{LiFu,LiYi2022} for minimization of energy dissipation via the stabilized method. An efficient iterative thresholding method \cite{LWCW2022,CLWW2022} was developed for topology optimization for Stokes and the Navier–Stokes flow. To the best of our knowledge, however, there exists no research work on energy dissipative gradient flow schemes for topology optimization constrained by the nonlinear partial differential equations such as the incompressible Navier-Stokes flow.

In this paper, we derive the gradient flow of the phase field model for topology optimization of incompressible Navier-Stokes flow and prove the energy dissipation property in continuous space by overcoming the introduction of the extra adjoint variables that are induced by nonlinear constraints. Then we propose energy dissipative gradient flow schemes (Allen-Cahn and Cahn-Hilliard types) based on the so-called stabilization method \cite{Shen1999, Shen2010} by adding stabilization terms to avoid strict time step constraints, which treats the nonlinear terms explicitly and the linear terms implicitly. For the Navier-Stokes flows by the porous medium approach, the Fr\'{e}chet differentiability of state functions with respect to phase field function \cite{Garcke2016} defines the unique solution of linearized Navier-Stokes equation. The uniform boundedness of the velocity field and pressure function is also guaranteed. These analysis results motivate us to deduce the  Lipschitz continuity of the solution to the phase field function which is of use in showing energy dissipation of the gradient flow. We note that the unconditional energy stability holds when the stabilized coefficients are larger than the coefficients depending on Lipschiz conditions, uniform bounds of the solution, and geometric area.

The rest of the paper is organized as follows:  In section 2, we briefly introduce the phase field model and the governing steady-state Navier-Stokes equations. Then the topology optimization is built. The uniform boundedness of the solution pairs and adjoint variables are shown by the assumption that the gradient of velocity is smaller than a prescribed constant depending on the viscous coefficient and geometric measure. In section 3, we proposed the generalized gradient flow schemes first to solve the topology optimization problem. The energy dissipation of the gradient flow in continuous space is addressed by combining both state and adjoint variables which is an important clue for the stability of the scheme in time discretization. Then the Lipschtiz continuity of both state and adjoint variables is shown via the error estimate analysis and Sobolev compact embedding Theorem. After that, we present the main Theorem to show the stabilized gradient flow schemes are unconditional energy stable for both Allen-Cahn and Chan-Hilliard types. We also prove the cut-off technique does not affect the monotonicity of the cost functional by gradient flow scheme. In section 4, we introduce the conforming mixed finite element method to discretize the state variables (velocity field, pressure function) as well as the adjoint variables. The algorithm with a stable semi-implicit scheme is proposed to evolve the phase field function numerically. Various numerical experiments in 2D and 3D have been tested to verify the effectiveness of the algorithms proposed. Section 4 draws brief conclusions and some potential values to generalize such gradient flow schemes to other topology optimization. 

\section{Model problem}
In this section, we introduce the phase field model for shape and topology optimization in a viscous incompressible fluid. The purpose of shape and topology optimization here is to seek an optimized configuration attaining the minimization of a given cost functional (typically dissipated energy) subject to stationary incompressible Navier-Stokes equations and some geometric constraint. 

Let $\Omega \subset \mathbb{R}^d\ (d= 2,3)$ be an open bounded domain with Lipschitz continuous boundary $\partial \Omega$. The whole domain $\Omega$ is partitioned into three disjoint subregions $\overline{\Omega}= \overline{\Omega}_1\cup \overline{\Omega}_2 \cup \overline{\Omega}_0$ where $\Omega_1$, $\Omega_2$ and $\Omega_0$ represent the fluid region, solid region and diffuse layer, respectively. The boundary $\partial\Omega$ is partitioned into both nonoverlapping Dirichlet and Neumann boundaries with $\partial\Omega=\Gamma_d\cup \Gamma_n$ (see Fig. \ref{fig1} left), where $\Gamma_d$ consists of the inlet as well as the wall, and $\Gamma_n$ represents the outlet. First, we introduce notations involving Sobolev spaces \cite{Adam}. Let $L^2(\Omega)$ be a Lebesgue space of square-integrable functions on $\Omega$. Denote $W^{1,2}(\Omega):=\{v\in L^2(\Omega)\,|\,D_i v\in L^2(\Omega),\, i=1,\cdots,d\}$ with $D_i v$ being the generalized derivative of $v$ with respect to $x_i$. Denote $H^1(\Omega):=W^{1,2}(\Omega)$ and $H_0^1(\Omega):=\{v\in H^1(\Omega)\,|\, v=0\ {\rm on}\ \partial\Omega\}$. Then denote the vectorial function spaces $\textbf{H}^1(\Omega):= H^1(\Omega)^d$, $\textbf{H}^1_0(\Omega):= H^1_0(\Omega)^d$ with its dual space $\textbf{H}^{-1}(\Omega)$, $\textbf{L}^2(\Omega):=L^2(\Omega)^d$ and  $\textbf{L}^\infty (\Omega):=L^\infty (\Omega)^d$. Denote Sobolev space with divergence-free constraint $\textbf{H}^1({\rm div0},\Omega):=\{\vw\in \textbf{H}^1(\Omega)\,|\, \nabla\cdot\vw =0 \ {\rm in}\ \Omega\}$. Let us use a same notation to define inner products of $L^2$ type by $(\zeta_1, \zeta_2):=\int_{\Omega} \zeta_1 \zeta_2$ and $(\bm\zeta_1, \bm\zeta_2):=\int_{\Omega} \bm\zeta_1\cdot \bm\zeta_2$ for any scalar functions $ \zeta_1,\zeta_2\in L^2(\Omega)$ and vectorial functions $\bm\zeta_1,\bm\zeta_2\in \textbf{L}^2(\Omega)$.
The phase field function $\phi\in H^1(\Omega)$ can be seen as a ``density" function (see Fig. \ref{fig1} right) such that 
\begin{equation}
\left\{
\begin{aligned}
&\phi(\bm x) = 0, \quad && \bm x\in \Omega_1,\\
&\phi(\bm x) = 1, \quad && \bm x\in \Omega_2,\\
&0<\phi(\bm x)<1, \quad && \bm x\in \Omega_0.
\end{aligned}
\right.
\end{equation}
\begin{figure}[htbp]
\centering
\begin{tikzpicture}[scale=3.2]
\draw[thick] (0.0,0.0) -- (1.5,0.0) -- (1.5,1.0) -- (0.0,1.0) -- cycle;
\draw[thick, fill=gray!30!white]  (0.0,0.85) .. controls (0.45,0.8).. (0.75, 0.7) .. controls (1.05,0.8).. (1.5, 0.85) -- (1.5,0.65) .. controls (1.35,0.6) .. (0.9,0.5) .. controls (1.35,0.4) .. (1.5,0.35) -- (1.5,0.15).. controls (1.05,0.2)..(0.75, 0.3).. controls (0.45,0.2).. (0.0,0.15) -- (0.0,0.35) 
.. controls (0.15,0.4) .. (0.6,0.5)  .. controls (0.15,0.6) .. (0.0,0.65) -- cycle;
\draw (0.8, 0.375) node[scale=0.8] {$\Omega_2(\phi=1)$};
\draw (1.15,0.07) node[scale=0.8] {$\Omega_1(\phi=0)$};
\draw (1.4,0.75) node[scale=0.8] {$\Gamma_n$};
\draw (1.4,0.25) node[scale=0.8] {$\Gamma_n$};
\draw (-0.1,0.5) node[scale=0.8] {$\Gamma_d$};
\draw (1.4,0.5) node[scale=0.8] {$\Gamma_d$};
\draw (0.75,1.05) node[scale=0.8] {$\Gamma_d$};
\draw (0.75,-0.05) node[scale=0.8] {$\Gamma_d$};
\end{tikzpicture}
\qquad\qquad
\begin{tikzpicture}[scale=3.15]
\draw[thick] (0.0,0.0) -- (1.0,0.0) -- (1.0,1.0) -- (0.0,1.0) -- cycle;
\draw[thick] (0.45,0.0) .. controls (0.65,0.5) .. (0.55,1.0);
\draw[thick, dashed, postaction={pattern=north east lines}] (0.5,0.0) .. controls (0.7,0.5) .. (0.6,1.0) -- (0.5, 1.0) .. controls (0.6, 0.5).. (0.4,0.0) -- cycle;
\draw (0.25,0.5) node[scale=1] {Phase A};
\draw (0.8,0.25) node[scale=1] {Phase B};
\draw (0.2,0.75) node[scale=1] {$\Omega$};
\draw [arrows = {-Latex[width=3pt, length=3pt]}]   (0.4,-0.025)--(0.5,-0.025);
\draw [arrows = {-Latex[width=3pt, length=3pt]}]   (0.5,-0.025)--(0.4,-0.025);
\draw (0.5,-0.075) node[scale=0.75] {$0<\phi<1$};
\end{tikzpicture}
\caption{Illustrations of design domain with subdomains represented implicitly by phase field function (left) and phases with diffuse layer (right).}
\label{fig1}
\end{figure}
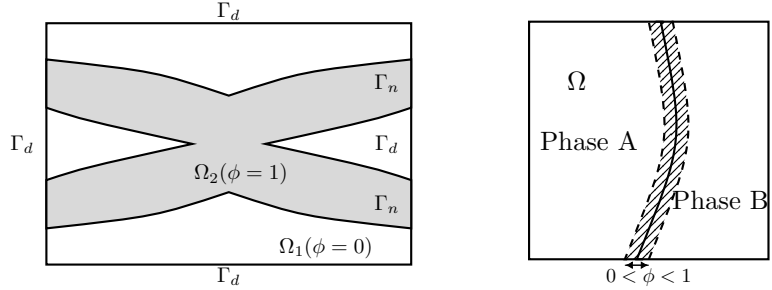
\subsection{Governing equations}
Let the positive number $\mu$ be the viscous coefficient. Define the following bilinear form $a(\cdot , \cdot)$ and trilinear form $b(\cdot, \cdot, \cdot)$, respectively 
\begin{equation*}
\begin{aligned}
&a(\vu,\vvv) := \mu \int_{\Omega}{\rm D} \vu :{\rm D} \vvv &&\forall \, \vu , \vvv \in \textbf{H}^1(\Omega),\\
&b(\vu,\vvv,\vw) :=\int_{\Omega}(\vu \cdot\nabla) \vvv \cdot \vw   \quad &&\forall \, \vu, \vvv, \vw \in \textbf{H}^1(\Omega),
\end{aligned}
\end{equation*}
where the vectorial functions are given by $\vu = [u_1,u_2,\cdots, u_d]^{\rm T}$, $\vvv = [v_1,v_2,\cdots, v_d]^{\rm T}$ and $\vw = [w_1,w_2,\cdots, w_d]^{\rm T}$, respectively. We refer to \cite[Lemma IX.1.1 and Lemma IX.2.1]{Galdi2011} and \cite{EMJ} for some useful properties on the trilinear form.
\begin{Lemma}\label{trlinearLemma}
The trilinear form $b(\cdot, \cdot, \cdot)$ is well-defined and continuous in the space $\textbf{H}^1_0(\Omega)\times \textbf{H}^1(\Omega)\times\textbf{H}^1_0(\Omega)$. The following estimate holds
\begin{equation}
\vert b(\vu,\vvv,\vw)\vert \leq K_{\Omega} \| \nabla\vu\|_{\textbf{L}^2(\Omega)} \| \nabla\vvv\|_{\textbf{L}^2(\Omega)} \| \nabla\vw\|_{\textbf{L}^2(\Omega)},\qquad \forall \vu,\vw\in \textbf{H}^1_0(\Omega),\  \vvv\in \textbf{H}^1(\Omega),
\end{equation}
where
\begin{equation*}
K_{\Omega}:=\left\{
\begin{aligned}
&\frac{1}{2}\vert \Omega\vert^{1/2},\quad & {\rm if}\ d=2,\\
&\frac{2\sqrt{2}}{3}\vert \Omega\vert^{1/6},\quad & {\rm if}\ d=3,
\end{aligned}\right.
\end{equation*}
with $|\Omega|$ being the Lebesgue measure of $\Omega$. Furthermore, the following properties hold:
\begin{equation}
\begin{aligned}
&b(\vu,\vvv,\vvv)=0, &&\forall\,\vu\in \textbf{H}^1({\rm div0},\Omega),\ \vvv\in \textbf{H}_0^1(\Omega),\\
&b(\vu,\vvv,\vw)=-b(\vu,\vw,\vvv),&&\forall\,\vu\in\textbf{H}^1({\rm div0},\Omega),\vvv,\vw\in \textbf{H}_0^1(\Omega).
\end{aligned}
\end{equation}
\end{Lemma}
The permeability function for medium \cite[page 225]{Garcke2016} $\alpha(\cdot)\in C^{1,1}(\mathbb{R})$ is non-negative with $\alpha(1)=0, \alpha(0)=\alpha_0>0$ satisfying
\begin{equation}
\begin{aligned}
& \alpha(s)=\alpha(0),\quad \forall\ s\leq 0,\\
& \alpha(s)=\alpha(1),\quad \forall\ s\geq 1.
\end{aligned}
\end{equation}
Let the Cauchy stress tensor be $\sigma(\vu,p):=-p{\rm \bm I}+\mu({\rm D}\vu+{\rm D}\vu^{\rm T})$ with ${\rm \bm{I}}\in\mathbb{R}^{d\times d}$ being an identity tensor. Let $\vu_d: \Gamma_d\rightarrow \mathbb{R}^d$ be a prescribed velocity field imposed on the inlet and the wall. Define the Sobolev spaces $\textbf{H}^1_d (\Omega):=\{\vvv\in \textbf{H}^1(\Omega)\,|\, \vvv = \vu_d \ {\rm on}\ \Gamma_d, \sigma(\vvv,p)\bm n =\bm{0}\ {\rm on}\ \Gamma_n \}$ and $\textbf{H}^1_{00} (\Omega):=\{\vvv\in \textbf{H}^1(\Omega)\,|\, \vvv = 0 \ {\rm on}\ \Gamma_d, \sigma(\vvv,q)\bm n =0\ {\rm on}\ \Gamma_n \}$ for state and adjoint variables. Consider a weak formulation of a steady-state incompressible Navier-Stokes equation: find $(\vu,p) \in \textbf{H}^1_d (\Omega)\times L^2(\Omega)$ such that
\begin{equation}\label{nsweak}
\left\{
\begin{aligned}
&a(\vu,\vvv)+b(\vu,\vu,\vvv) + (\alpha(\phi)\vu,\vvv)-(p, \nabla \cdot\vvv)= (\vf,\vvv) \quad  &&\forall\ \vvv \in \textbf{H}^1_{00}(\Omega), \\
&(\nabla \cdot\vu,q)  = 0\quad &&\forall\ q \in L^2(\Omega),
\end{aligned}\right.
\end{equation}
where $\vf \in \textbf{L}^2(\Omega)$ is a given source. The existence of (\ref{nsweak}) is valid (see \cite[Lemma 2.7]{GarckeCOCV}) for every phase field function $\phi\in L^1(\Omega)$ with $\vert \phi \vert \leq 1$ a.e. in $\Omega$. The uniqueness \cite[Lemma 2.8]{GarckeCOCV} of (\ref{nsweak}) holds by the assumption that $\Vert \nabla \vu \Vert_{\textbf{L}^2(\Omega)}< \frac{\mu}{K_{\Omega}}$. 

\subsection{Phase-field model}
The shape functional with specific purpose is defined by the phase field function and the velocity field $J(\phi,\vu(\phi)):=\int_{\Omega} j(\phi,\vu(\phi))$ where the non-negative integrand $j(\phi,\vu(\phi))$ is Fr\'{e}chet differentiable with respect to $\phi$ and $\vu$. Consider the shape functional of energy dissipation $j(\phi,\vu(\phi)):=\frac{1}{2}\mu\vert {\rm D}\vu\vert^2 +\frac{1}{2}\alpha(\phi)|\vu|^2$. Let $V(\phi):=\int_{\Omega}(1-\phi)$ be the volume of the solid region. The double well potential \cite[page 421]{JSphasefieldNM}
\begin{equation}
\omega(\phi)=\left\{
\begin{aligned}
&\phi^2,\quad &&\phi<0,\\
&\frac{1}{4}\phi^2(\phi-1)^2,\quad && \phi\in [0,1],\\
&(\phi-1)^2, && \phi>1.
\end{aligned}\right.
\end{equation}
has the formulation with uniform upper bound of its derivative.
Next, we can construct the total free energy by the summation of the Ginzburg-Landau energy, the shape functional and the least square of volume error as
\begin{equation}\label{TopPro}
\mathcal{W}(\phi,\vu(\phi)):= \frac{\epsilon_1}{2}\|\nabla\phi\|_{\mathrm{L^2(\Omega)}}^2+\frac{1}{\epsilon_2}(\omega(\phi),1)+ J(\phi,\vu(\phi))+\frac{1}{2}\beta(V(\phi)-\hat{V})^2,
\end{equation}
where $\epsilon_i >0\, (i=1,2)$ represents the thickness of the diffuse layer, $\hat{V}>0$ and $\beta>0$. Then the shape and topology optimization problem is to seek $\phi^*\in H^1(\Omega)$ such that
\begin{equation}\label{OptControlPro}
\mathcal{W}(\phi^*,\vu(\phi^*))=\inf_{\phi\in H^1(\Omega), 0\leq \phi \leq 1} \mathcal{W}(\phi,\vu(\phi)),
\end{equation}
where $(\vu(\phi), p(\phi))$ is the solution of (\ref{nsweak}) with given phase field function $\phi$. The existence of the minimizer for optimal control problem (\ref{OptControlPro}) holds by using the lower semi-continuity of the objective functional and some compactness properties (see \cite{Garcke2016,GarckeCOCV}). 
We introduce the following adjoint equations corresponding to the optimization problem (\ref{OptControlPro}).
\begin{Lemma}\label{adjLemma}
Let domain $\Omega$ be an open bounded domain. Suppose that $j(\phi,\vu)$ is differentiable with respect to $\vu$ with the Fr\'{e}chet derivative denoted as $j_{\vu}:=\frac{\delta j(\phi,\vu)}{\delta \vu}$. Then the weak form of the adjoint problem (a generalized Stokes equation) satisfies: find $(\vvv,q)\in \textbf{H}^1_{00}(\Omega)\times L^2(\Omega)$ such that
\begin{equation}\label{adjequ}\left\{
\begin{aligned}
&a(\vvv,\vw)+b(\vw,\vu,\vvv) + b(\vu,\vw,\vvv)+ (\alpha(\phi) \vvv, \vw)-(q, \nabla \cdot \vw)=(j_{\vu}, \vw) &&\ \forall \,\vw \in \textbf{H}^1(\Omega),\\
& (z, \nabla\cdot \vvv) =0, && \ \forall \,z \in L^2(\Omega),
\end{aligned}\right.
\end{equation}
where the directional derivative $(j_{\vu}, \vw) = \mu({\rm D}\vu, {\rm D}\vw)+(\alpha(\phi)\vu, \vw)$ for energy dissipation.
\end{Lemma}
\begin{proof}
By utilizing the weak formulation of (\ref{nsweak}), define the following functional
\begin{equation}\label{adjointStp1}
\mathcal{L}(\phi,\vu,p;\vvv,q) := a(\vu,\vvv)+b(\vu,\vu,\vvv) + (\alpha(\phi)\vu, \vvv)-(p, \nabla \cdot\vvv)-(q, \nabla \cdot\vu)- (\vf,\vvv).
\end{equation}
The Lagrange functional is introduced associated with total free energy $\mathcal{W}(\phi,\vu)$ and (\ref{adjointStp1})
\begin{equation}\label{adjointStp2}
\hat{\mathcal{L}}(\phi,\vu,p;\vvv,q) := \mathcal{W}(\phi,\vu)-\mathcal{L}(\phi,\vu,p;\vvv,q).
\end{equation}
Then the constrained problem can be transformed into the saddle point problem \cite{Defour}
\begin{equation}
\inf_{\phi\in H^1(\Omega)} \underbrace{\mathcal{W}(\phi,\vu)}_{(\vu,p)\ \text{satisfy}\ (\ref{nsweak})} =\inf_{\phi\in H^1(\Omega)} \inf_{(\vu,p) \in \textbf{H}^1_d(\Omega)\times L^2(\Omega)}\sup_{(\vvv,q)\in \textbf{H}^1_{00}(\Omega)\times L^2(\Omega)} \hat{\mathcal{L}}(\phi, \vu,p;\vvv,q).
\end{equation}
By the Karusch-Kuhn-Tucker condition, the saddle point of $\hat{\mathcal{L}}$ is characterized by
\begin{equation}\label{adjointStp3}
\begin{aligned}
&\frac{\delta \hat{\mathcal{L}}}{\delta\vu}( d_{\vu})=\frac{\delta \hat{\mathcal{L}}}{\delta p}( d_p)=0,\quad &&\forall\ d_{\vu}\in \textbf{H}^1(\Omega), d_p \in L^2(\Omega), \\
&\frac{\delta \hat{\mathcal{L}}}{\delta\vvv}( d_{\vvv})=\frac{\delta \hat{\mathcal{L}}}{\delta q}( d_q)=0,&&\forall\ d_{\vvv}\in \textbf{H}^1(\Omega), d_q \in L^2(\Omega).
\end{aligned}
\end{equation}
The first line in (\ref{adjointStp3}) implies the adjoint equations while the second line in (\ref{adjointStp3}) leads to the steady-state incompressible Navier-Stokes equations.
\end{proof}

The following analysis is based on the assumption that the boundary condition of the velocity field is the homogeneous Dirichlet condition. 
\begin{Lemma}
Let $\phi\in H^1(\Omega)\cap L^\infty(\Omega)$ be given such that $\|\nabla\vu\|_{\textbf{L}^2(\Omega)}\leq \frac{\mu}{2K_{\Omega}}$. Then, there exists a unique solution pair $(\vu, p)\in \textbf{H}^1(\Omega)\times L^2_0(\Omega)$ of the Navier-Stokes equations (\ref{nsweak}) and a unique solution pair $(\vvv, q)\in \textbf{H}^1_{0}(\Omega)\times L^2_0(\Omega)$ of the adjoint equations (\ref{adjequ}). The solution pairs fulfill the following estimates
\begin{equation}\label{upbound}
    \| \vu\|_{\textbf{H}^1(\Omega)}+\| p\|_{L^2(\Omega)}\leq \mathcal{C}_1 (\mu, \alpha_0, \bm f, \Omega),
\end{equation}
and
\begin{equation}\label{vqbound}
    \| \vvv\|_{\textbf{H}^1(\Omega)}+\| q\|_{L^2(\Omega)}\leq \mathcal{C}_2 (\mu, \alpha_0, \bm f, \Omega),
\end{equation}
with constants $\mathcal{C}_1 (\mu, \alpha_0, \bm f, \Omega)$ and $\mathcal{C}_2 (\mu, \alpha_0, \bm f, \Omega)$ independent of $\phi$.
\end{Lemma}
\begin{proof}
We refer to  \cite[Lemma 2.8]{GarckeCOCV} where the existence and uniqueness results for the Navier-Stokes equations (\ref{nsweak}) are discussed requiring that $\|\nabla\vu\|_{\textbf{L}^2(\Omega)}\leq \frac{\mu}{K_{\Omega}}$. The uniform boundedness (\ref{upbound}) of the solution pair for Navier-Stokes equations (\ref{nsweak}) independent of phase field function $\phi$ is proved referring to \cite[Lemma 4.3]{Garcke2016}. For the adjoint equations (\ref{adjequ}) of the general cost functional, the solvability and uniqueness has also been discussed referring to \cite[Lemma 4.9]{Garcke2016} requiring $\|\nabla\vu\|_{\textbf{L}^2(\Omega)}\leq \frac{\mu}{K_{\Omega}}$. We are going to show the uniform boundedness of the adjoint variables by assuming $\|\nabla\vu\|_{\textbf{L}^2(\Omega)}\leq \frac{\mu}{2K_{\Omega}}$ which is of help to the following estimate. Let the test function $\bm w=\vvv$ in (\ref{adjequ}) so that $b(\vu,\vvv,\vvv)=0$ for $\nabla\vu=0$. After applying Lemma \ref{trlinearLemma}, Cauchy inequality and Poincar\'{e} inequality, we obtain
\begin{equation}
\begin{aligned}
a(\vvv,\vvv)+(\alpha(\phi)\vvv,\vvv)=&-b(\vvv,\vu,\vvv)+(j_{\vu},\vvv)\\
\leq& K_\Omega \|\nabla\vu\|_{\textbf{L}^2(\Omega)}\|\nabla\vvv\|_{\textbf{L}^2(\Omega)}^2+\|j_{\vu}\|_{\textbf{L}^{2}(\Omega)}\|\vvv\|_{\textbf{L}^2(\Omega)}\\
\leq& K_\Omega \|\nabla\vu\|_{\textbf{L}^2(\Omega)}\|\nabla\vvv\|_{\textbf{L}^2(\Omega)}^2+C_p \|j_{\vu}\|_{\textbf{L}^{2}(\Omega)}\|\nabla\vvv\|_{\textbf{L}^2(\Omega)}\\
\leq& \frac{\mu}{2}\|\nabla\vvv\|_{\textbf{L}^2(\Omega)}^2+\frac{\mu}{8}\|\nabla \vvv\|_{\textbf{L}^2(\Omega)}^2+\frac{2C_p^2}{\mu}\| j_{\vu}\|_{\textbf{L}^2(\Omega)}^2,
\end{aligned}
\end{equation}
which implies that
\begin{equation}
\frac{3\mu}{8}\|\nabla \vvv\|_{\textbf{L}^2(\Omega)}^2\leq \frac{2C_p^2}{\mu}\| j_{\vu}\|_{\textbf{L}^2(\Omega)}^2.
\end{equation}
Furthermore by \cite[Lemma II.2.1.1]{Sohr}, we obtain unique $q\in L_0^2(\Omega)$ such that 
\begin{equation}
\|q\|_{L_0^2(\Omega)}\leq \mathcal{C}_0 \|-\mu\Delta\vvv+\alpha(\phi)\vvv+{\rm D}\vu^{\rm T}\vvv -{\rm D}\vvv \vu -j_{\vu} \|_{\textbf{H}^{-1}(\Omega)},
\end{equation}
is fulfilled for some constant $\mathcal{C}_0>0$. 
\end{proof}

\section{Gradient flow}
In this section, we construct an efficient and effective method to solve the optimal control problem (\ref{OptControlPro}) numerically. The minimum can be found by introducing the gradient flow with a virtual temporal dimension. The phase field function is extended to $\phi(t, \bm x)\in H^1(0,T; H^1(\Omega))$ where $T$ is the prescribed terminal time. Given free energy functional $\mathcal{W}(\phi,\vu(\phi))$ bounded from below, denote its variational derivative as $\nu = \frac{\delta \mathcal{W}}{\delta \phi}$. The general form of the gradient flow \cite{Shen2019} can be written as 
\begin{equation}
    \frac{\partial \phi}{\partial t} = \mathcal{G}\nu \quad {\rm in}\, (0,T]\times\Omega
\end{equation}
supplemented with suitable boundary conditions. The nonpositive symmetric operator $\mathcal{G}$ is the dissipation mechanism including the $L^2$ gradient flow of Allen-Cahn type with $\mathcal{G}=- {\rm I}$ and the $H^{-1}$ gradient flow of Cahn-Hilliard type with $\mathcal{G}=-\Delta$ (the Laplacian). Since the $H^{-1}$ gradient flow preserves mass conservation, hence no extra volumetric constraint needs to be introduced. To simplify the presentation, we assume throughout the paper that the boundary conditions are chosen such that all boundary terms will vanish when integration by parts is performed. The gradient flow for solving the optimal control problem (\ref{OptControlPro}) reads: find $\phi(t,\bm x) \in H^1(0,T; H^1(\Omega))$ such that
\begin{equation}\label{GradFlowContinuous}\left\{
\begin{aligned}
\frac{\partial \phi}{\partial t} &= \mathcal{G}\nu \quad {\rm in}\, (0,T]\times\Omega,\\
\nu &= -\epsilon_1 \Delta\phi + \frac{1}{\epsilon_2}\omega^\prime(\phi)+j_{\phi}(\phi, \vu(\phi))-\alpha^\prime(\phi)\vu\cdot\vvv + \beta(V(\phi)-\hat{V})V^\prime(\phi).
\end{aligned}\right.
\end{equation}
\begin{Lemma}\label{SenStructure}
Let the assumptions in Lemma \ref{adjLemma} hold. Let $(\vu,p)$ and $(\vvv,q)$ be the solution pairs of (\ref{nsweak}) and (\ref{adjequ}), respectively. Then the Fr\'{e}chet derivative holds as
\begin{equation}\label{senEqu}
 j^\prime(\phi,\vu(\phi)) = \frac{1}{2}\alpha^\prime (\phi)\vu(\phi)^2- \alpha^\prime (\phi)\vu(\phi)\cdot\vvv(\phi).
\end{equation}
\end{Lemma}
\begin{proof}
The existence of variational derivative $\vu_{\phi}:=\langle\vu^\prime(\phi), \zeta \rangle$ and $p_{\phi}:=\langle p^\prime(\phi), \zeta\rangle$ hold for all $\zeta \in H^1(\Omega)$ by the implicit Theorem \cite{Garcke2016} implying that
\begin{equation}\label{SenStep1}
a(\vu_{\phi},\vvv)+b(\vu_{\phi},\vu,\vvv)+b(\vu,\vu_{\phi},\vvv) + (\alpha(\phi)\vu_{\phi},\vvv)-(p_{\phi}, \nabla \cdot\vvv)-(q, \nabla \cdot \vu_{\phi})+(\alpha^\prime (\phi)\vu\cdot\vvv, \zeta)= 0.
\end{equation}
By the chain rule of variational differentiation, we obtain
\begin{equation}\label{SenStep2}
\begin{aligned}
\bigg\langle \frac{\delta J(\phi,\vu(\phi))}{\delta \phi}, \zeta\bigg\rangle&=\langle J_\phi (\phi,\vu(\phi)), \zeta \rangle+\bigg\langle \frac{\delta J(\phi,\vu(\phi))}{\delta \vu}, \vu_{\phi}\bigg\rangle \\
&=\langle J_\phi (\phi,\vu(\phi)), \zeta\rangle+(j_{\vu}, \vu_{\phi}).
\end{aligned}
\end{equation}
Taking the test function with $(\vu_{\phi},p_{\phi})$ in adjoint equations (\ref{adjequ}), we have
\begin{equation}\label{SenStep3}
a(\vvv,\vu_{\phi})+{b}(\vu_{\phi},\vu,\vvv) + b(\vu,\vu_{\phi},\vvv)+ (\alpha(\phi) \vvv, \vu_{\phi})-(p_{\phi}, \nabla \cdot\vvv)-(q, \nabla \cdot \vu_{\phi})=(j_{\vu}, \vu_{\phi}).
\end{equation}
Combing (\ref{SenStep1}), (\ref{SenStep2}) and (\ref{SenStep3}) yields
\begin{equation}
(j_{\vu}, \vu_{\phi}) =- (\alpha^\prime (\phi)\vu\cdot\vvv, \zeta),
\end{equation}
which allows the conclusion to hold.
\end{proof}

\begin{Proposition}\label{enerDisCon}
Under the gradient flow (\ref{GradFlowContinuous}), the energy stable holds for all $t>0$ such that
\begin{equation}\label{jojew}
    \frac{\delta \mathcal{W}}{\delta t} = (\nu, \mathcal{G}\nu)\leq 0.
\end{equation}
\end{Proposition}
\begin{proof}
Take the differentiation of the total energy in (\ref{TopPro}) with respect to the temporal variable $t$ yielding that
\begin{equation}\label{EneryDecayConS1}
\frac{\delta \mathcal{W}}{\delta t} = (\epsilon \nabla \phi, \nabla \phi_t) + (\omega^\prime (\phi),\phi_t)+\beta(V(\phi)-\hat{V})(V^{\prime}(\phi),\phi_t) + \bigg( \frac{\delta j(\phi,\vu(\phi))}{\delta \phi}, \phi_t\bigg),
\end{equation}
where the derivative denotes $\phi_t:= \frac{\partial \phi}{\partial t}$. Next for the last term in (\ref{EneryDecayConS1}), we obtain 
\begin{equation}\label{EneryDecayConS2}
\begin{aligned}
\bigg( \frac{\delta j(\phi,\vu(\phi))}{\delta \phi}, \phi_t\bigg) &=(j_\phi ,\phi_t) + (j_{\vu}, \vu_t),
\end{aligned}
\end{equation}
according to the chain rule in differentiation $(\vu_t, p_t):= (\vu_\phi \frac{\partial \phi}{\partial t},p_\phi \frac{\partial \phi}{\partial t})$. After using the adjoint equations (\ref{adjequ}) with test functions $(\vu_t, p_t)$, we have
\begin{equation}\label{EneryDecayConS3}
a(\vvv,\vu_t)+b(\vu_t,\vu,\vvv) + b(\vu,\vu_t,\vvv)+ (\alpha(\phi) \vvv, \vu_t)-(q, \nabla \cdot \vu_t)-(p_t, \nabla \cdot \vvv)=(j_{\vu}, \vu_t).
\end{equation}
Differentiate (\ref{nsweak}) with respect to $t$ to derive
\begin{equation}\label{EneryDecayConS4}
a(\vu_t,\vvv)+b(\vu_t,\vu,\vvv) +b(\vu,\vu_t,\vvv)+ (\alpha(\phi)\vu_t,\vvv)+ (\alpha^\prime (\phi)\vu\cdot \vvv,\phi_t)-(p_t, \nabla \cdot\vvv) - (q,\nabla\cdot \vu_t)= 0.
\end{equation}
Combining (\ref{EneryDecayConS2}) - (\ref{EneryDecayConS4}), we obtain
\begin{equation}\label{EneryDecayConS5}
    \bigg( \frac{\delta j(\phi,\vu(\phi))}{\delta \phi}, \phi_t\bigg) = \big(j_{\phi}-\alpha^\prime (\phi)\vu\cdot \vvv , \phi_t\big).
\end{equation}
By using the integration by parts and combining (\ref{GradFlowContinuous}), (\ref{EneryDecayConS1}) and (\ref{EneryDecayConS5}), we obtain \eqref{jojew}
thanks to the nonnegative operator of $\mathcal{G}$.
\end{proof}

Now, we are going to construct the energy dissipative gradient flow scheme for solving the optimal control problem (\ref{OptControlPro}).  Let $T>0$ and $0=t_0 < t_1<\cdots < t_n<t_{n+1}<\cdots< t_N=T$ denote a time partition with time step size $\tau_n:=t_{n+1}-t_{n}$ $(n=0,1,\cdots,N-1)$ for $N\in\mathbb{N}$. For simplicity, consider a uniform time discretization with $t_n=n\tau$, where the number of time levels $N=T / \tau$. The first-order semi-implicit scheme with generalized stabilization of gradient flow (\ref{GradFlowContinuous}) reads
\begin{equation}\label{GradFlowScheme}\left\{
\begin{aligned}
\frac{\phi^{n+1}-\phi^n}{\tau} &= \mathcal{G}\nu^{n+1},\\
\nu^{n+1} &= -\epsilon_1 \Delta\phi^{n+1} + \mathcal{U}(\phi^n, \vu^n, \vvv^n) + \mathcal{S}(\phi^{n+1}-\phi^n),
\end{aligned}\right.
\end{equation}
where the nonlinear term 
$$\mathcal{U}(\phi^n, \vu^n, \vvv^n):= \frac{1}{\epsilon_2} \omega^\prime(\phi^{n}) +j_\phi(\phi^{n},\vu^{n})-\alpha^\prime(\phi^{n})\vu^{n}\cdot\vvv^{n} + \beta(V(\phi^{n})-\hat{V})V^\prime(\phi^{n}),$$
and the general stabilization operator denotes $\mathcal{S}= S_0 + S_1(-\Delta)$ with each $S_0, S_1>0$. Then we begin to prove the property of the unconditional energy stable for the semi-implicit scheme (\ref{GradFlowScheme}). For proving the unconditional energy stability of the scheme (\ref{GradFlowScheme}), we need to verify the Lipschtiz boundedness of state and adjoint variables with respect to the phase field function 
\begin{Lemma}\label{boundLemma}
Let $\phi\in H^1(\Omega)\cap L^\infty(\Omega)$ be given such that $\|\nabla\vu\|_{\textbf{L}^2(\Omega)}\leq \frac{\mu}{2K_{\Omega}}$. Suppose that $\alpha(\phi)$ is Lipschitz continuous with respect to its argument: For any $\phi^{n+1},\phi^n\in H^1(\Omega)\cap L^\infty(\Omega)$, there exists positive coefficient $\gamma>0$ independent of $\phi$ satisfying that
\begin{equation}
\| \alpha(\phi^{n+1})-\alpha(\phi^n)\|_{L^2(\Omega)}\leq \gamma \| \phi^{n+1}-\phi^n\|_{L^2(\Omega)}.
\end{equation}
Suppose that $(\phi^n,\vu^n,p^n)$ and $(\phi^{n+1},\vu^{n+1},p^{n+1})$ are the solution of Navier-Stokes equations (\ref{nsweak}), respectively. Then the following estimate holds for state variables
\begin{equation}\label{subUest}
\|\nabla(\vu^{n+1}-\vu^n)\|_{\textbf{L}^2(\Omega)}\leq \tilde{C}_1\| \phi^{n+1}-\phi^n\|_{L^2(\Omega)},
\end{equation}
where
$\tilde{C}_1:={\gamma(1+C_p)^2}/{K_{\Omega}}$.
Furthermore, if $(\phi^n,\vvv^n,q^n)$ and $(\phi^{n+1},\vvv^{n+1},q^{n+1})$ are the solution of adjoint equations (\ref{adjequ}), respectively. Then the following estimates hold for adjoint variables
\begin{equation}
\|\nabla(\vvv^{n+1}-\vvv^n)\|_{\textbf{L}^2(\Omega)}\leq \tilde{C}_2\| \phi^{n+1}-\phi^n\|_{L^2(\Omega)},
\end{equation}
where
\begin{equation}
\tilde{C}_2:= \max \big\{ 2K_{\Omega} \tilde{C}_1 \mathcal{C}_2 , \gamma(1+C_p)^2(\mathcal{C}_1+\mathcal{C}_2), (\alpha_0 C_p^2 \tilde{C}_1+\mu\tilde{C}_1) \big\}
\end{equation}
with $C_p$ being a constant related to the Poincar\'{e} inequality. 
\end{Lemma}
\begin{proof}
Let $(\phi^n,\vu^n,p^n)$ and $(\phi^{n+1},\vu^{n+1},p^{n+1})$ be the solution pairs of Navier-Stokes equations (\ref{nsweak}), respectively. Then the substraction together with setting the test function $\vvv=\vu^{n+1}-\vu^n$ yields that
\begin{equation}
\begin{aligned}
&a(\vu^{n+1}-\vu^n,\vu^{n+1}-\vu^n)+b(\vu^{n+1}-\vu^n,\vu^{n+1},\vu^{n+1}-\vu^n)\\
&+b(\vu^n,\vu^{n+1}-\vu^n,\vu^{n+1}-\vu^n)+([\alpha(\phi^{n+1})-\alpha(\phi^n)]\vu^{n+1},\vu^{n+1}-\vu^n)\\ & +(\alpha(\phi^n)(\vu^{n+1}-\vu^n),\vu^{n+1}-\vu^n)=0.
\end{aligned}
\end{equation}
For $\vu^n\in \textbf{H}^1({\rm div0},\Omega)$, we have $b(\vu^n,\vu^{n+1}-\vu^n,\vu^{n+1}-\vu^n)=0$. After applying Lemma \ref{trlinearLemma}, Cauchy inequality, Hölder equality, and Poincar\'{e} inequality, we obtain
\begin{equation}
\begin{aligned}
&\mu \| \nabla(\vu^{n+1}-\vu^n)\|^2_{\textbf{L}^2(\Omega)}+(\alpha(\phi^n)(\vu^{n+1}-\vu^n),\vu^{n+1}-\vu^n)\\
\leq& K_{\Omega} \|\nabla(\vu^{n+1}-\vu^n)\|_{\textbf{L}^2(\Omega)}^2 \| \nabla\vu^{n+1}\|_{\textbf{L}^2(\Omega)}\\
&+\| \alpha(\phi^{n+1})-\alpha(\phi^n)\|_{L^2(\Omega)}\| \vu^{n+1}\|_{\textbf{L}^4(\Omega)}\| \vu^{n+1}-\vu^n\|_{\textbf{L}^4(\Omega)}\\
\leq& \frac{\mu}{2}\|\nabla(\vu^{n+1}-\vu^n)\|_{\textbf{L}^2(\Omega)}^2\\
&+\gamma(1+C_p)^2\|\phi^{n+1}-\phi^n\|_{L^2(\Omega)}\| \nabla\vu^{n+1}\|_{\textbf{L}^2(\Omega)}\| \nabla(\vu^{n+1}-\vu^n)\|_{\textbf{L}^2(\Omega)},
\end{aligned}
\end{equation}
where the Sobolev imbedding Theorem is used for $\textbf{H}^1(\Omega)\hookrightarrow \textbf{L}^4(\Omega)$ ignoring its constant and $C_p$ depends on the Poincar\'{e} inequality. After that, we can deduce the following estimate
\begin{equation}
\frac{\mu}{2}\|\nabla(\vu^{n+1}-\vu^n)\|_{\textbf{L}^2(\Omega)}\leq \gamma(1+C_p)^2\frac{\mu}{2K_{\Omega}}\| \phi^{n+1}-\phi^n\|_{L^2(\Omega)}.
\end{equation}
Let $(\phi^n,\vvv^n,q^n)$ and $(\phi^{n+1},\vvv^{n+1},q^{n+1})$ are the solution pairs of adjoint equations (\ref{adjequ}). The substraction by setting the test function $\vw=\vvv^{n+1}-\vvv^n$ yields that
\begin{equation}
\begin{aligned}
&a(\vvv^{n+1}-\vvv^n,\vvv^{n+1}-\vvv^n)+b(\vvv^{n+1}-\vvv^n, \vu^{n+1},\vvv^{n+1})-b(\vvv^{n+1}-\vvv^n, \vu^{n},\vvv^{n})\\
&+b(\vu^{n+1},\vvv^{n+1}-\vvv^n,\vvv^{n+1})-b(\vu^{n},\vvv^{n+1}-\vvv^n,\vvv^{n})\\
&+(\alpha(\phi^{n+1})\vvv^{n+1},\vvv^{n+1}-\vvv^n)-(\alpha(\phi^{n})\vvv^{n},\vvv^{n+1}-\vvv^n)\\
=&a(\vu^{n+1}-\vu^n, \vvv^{n+1}-\vvv^n)+(\alpha(\phi^{n+1})\vu^{n+1}-\alpha(\phi^{n})\vu^{n}, \vvv^{n+1}-\vvv^n).
\end{aligned}
\end{equation}
By the rearrangement of the trilinear terms, we obtain
\begin{equation}
\begin{aligned}
&a(\vvv^{n+1}-\vvv^n,\vvv^{n+1}-\vvv^n)+b(\vvv^{n+1}-\vvv^n, \vu^{n+1}-\vu^n,\vvv^{n+1})+b(\vvv^{n+1}-\vvv^n, \vu^{n},\vvv^{n+1}-\vvv^{n})\\
&+b(\vu^{n+1}-\vu^n,\vvv^{n+1}-\vvv^n,\vvv^{n+1})+b(\vu^{n},\vvv^{n+1}-\vvv^n,\vvv^{n+1}-\vvv^n)\\
&+([\alpha(\phi^{n+1})-\alpha(\phi^{n})]\vvv^{n+1},\vvv^{n+1}-\vvv^n)+(\alpha(\phi^{n})(\vvv^{n+1}-\vvv^{n}),\vvv^{n+1}-\vvv^n)\\
=&a(\vu^{n+1}-\vu^n, \vvv^{n+1}-\vvv^n)+([\alpha(\phi^{n+1})-\alpha(\phi^{n})]\vu^{n+1}, \vvv^{n+1}-\vvv^n)\\
&+(\alpha(\phi^{n})(\vu^{n+1}-\vu^{n}), \vvv^{n+1}-\vvv^n).
\end{aligned}
\end{equation}
For $\vu^{n}\in \textbf{H}^1({\rm div0},\Omega)$, we have $b(\vu^{n},\vvv^{n+1}-\vvv^n,\vvv^{n+1}-\vvv^n)=0$. Similarly applying Lemma \ref{trlinearLemma}, Cauchy inequality, Hölder equality and Poincar\'{e} inequality, we obtain
\begin{equation}
\begin{aligned}
&\mu\| \nabla(\vvv^{n+1}-\vvv^{n})\|_{\textbf{L}^2(\Omega)}^2+(\alpha(\phi^n)(\vvv^{n+1}-\vvv^n),\vvv^{n+1}-\vvv^n)\\
\leq &2K_{\Omega}\|\nabla(\vvv^{n+1}-\vvv^{n})\|_{\textbf{L}^2(\Omega)}\|\nabla(\vu^{n+1}-\vu^{n})\|_{\textbf{L}^2(\Omega)}\|\nabla \vvv^{n+1}\|_{\textbf{L}^2(\Omega)} \\
&+K_{\Omega}\|\nabla(\vvv^{n+1}-\vvv^{n})\|^2_{\textbf{L}^2(\Omega)}\|\nabla\vu^{n}\|_{\textbf{L}^2(\Omega)}\\
&+\| \alpha(\phi^{n+1})-\alpha(\phi^n)\|_{L^2(\Omega)}\| \vvv^{n+1}-\vvv^n \|_{\textbf{L}^4(\Omega)}(\|\vvv^{n+1}\|_{\textbf{L}^4(\Omega)}+\|\vu^{n+1}\|_{\textbf{L}^4(\Omega)})\\
&+\|\alpha(\phi^n)\|_{L^\infty(\Omega)}\|\vvv^{n+1}-\vvv^n\|_{\textbf{L}^2(\Omega)}\|\vu^{n+1}-\vu^n\|_{\textbf{L}^2(\Omega)} \\
&+\mu \| \nabla(\vvv^{n+1}-\vvv^{n})\|_{\textbf{L}^2(\Omega)}\| \nabla(\vu^{n+1}-\vu^{n})\|_{\textbf{L}^2(\Omega)},
\end{aligned}
\end{equation}
yielding
\begin{equation}
\begin{aligned}
&\mu\| \nabla(\vvv^{n+1}-\vvv^{n})\|_{\textbf{L}^2(\Omega)}^2\\
\leq &2K_{\Omega} \tilde{C}_1 \mathcal{C}_2\|\nabla(\vvv^{n+1}-\vvv^{n})\|_{\textbf{L}^2(\Omega)}\|\phi^{n+1}-\phi^n\|_{{L}^2(\Omega)} +\frac{\mu}{2}\|\nabla(\vvv^{n+1}-\vvv^{n})\|^2_{\textbf{L}^2(\Omega)}\\
&+\gamma(1+C_p)^2\| \phi^{n+1}-\phi^n\|_{L^2(\Omega)}\| \nabla(\vvv^{n+1}-\vvv^n) \|_{\textbf{L}^2(\Omega)}(\|\nabla\vvv^{n+1}\|_{\textbf{L}^2(\Omega)}+\|\nabla\vu^{n+1}\|_{\textbf{L}^2(\Omega)})\\
&+\alpha_0 C_p^2\|\nabla(\vvv^{n+1}-\vvv^n)\|_{\textbf{L}^2(\Omega)}\| \nabla(\vu^{n+1}-\vu^n)\|_{\textbf{L}^2(\Omega)}\\
&+\mu \| \nabla(\vvv^{n+1}-\vvv^{n})\|_{\textbf{L}^2(\Omega)}\| \nabla(\vu^{n+1}-\vu^{n})\|_{\textbf{L}^2(\Omega)},
\end{aligned}
\end{equation}
where the Sobolev imbedding Theorem is used for $\textbf{H}^1(\Omega)\hookrightarrow \textbf{L}^4(\Omega)$. The boundedness result can be further deduced by
\begin{equation}
\begin{aligned}
\frac{\mu}{2}\| \nabla(\vvv^{n+1}-\vvv^{n})\|_{\textbf{L}^2(\Omega)}
\leq & 2K_{\Omega} \tilde{C}_1 \mathcal{C}_2\|\phi^{n+1}-\phi^n\|_{{L}^2(\Omega)} + \gamma(1+C_p)^2(\mathcal{C}_1+\mathcal{C}_2) \|\phi^{n+1}-\phi^n\|_{{L}^2(\Omega)}\\
&+(\alpha_0 C_p^2 \tilde{C}_1+\mu\tilde{C}_1) \|\phi^{n+1}-\phi^n\|_{{L}^2(\Omega)}
\end{aligned}
\end{equation}
where the estimate (\ref{subUest}) is used.
\end{proof}

Next, the estimate of the adjacent cost functionals is discussed.
\begin{Lemma}\label{StaJ}
Suppose that the phase field function $\phi^{n+1}$ is evolved by the scheme (\ref{GradFlowScheme}).

\noindent
Let $(\phi^{n+1}, \vu^{n+1},p^{n+1})$ and $(\phi^n, \vu^{n},p^{n})$ be the solution pairs of Navier-Stokes equations (\ref{nsweak}). Furthermore, let $(\phi^{n+1},\vvv^{n+1},q^{n+1})$ and $(\phi^n,\vvv^{n},q^{n})$ be the solution pairs of adjoint equations (\ref{adjequ}). Then the subtraction of adjacent cost functional gives
\begin{equation}\label{subJinLEMMA}
\begin{aligned}
&J(\phi^{n+1},\vu^{n+1}) - J(\phi^{n},\vu^{n}) 
=\bigg(\frac{1}{2}\alpha^\prime(\phi^n)(\vu^n)^2-\alpha^\prime(\phi^n)\vu^n\cdot\vvv^{n},\phi^{n+1}-\phi^n \bigg)\\
&-\frac{\mu}{2}\int_{\Omega} |{\rm D}\vu^{n+1} - {\rm D}\vu^n|^2- \int_{\Omega} \frac{\alpha(\phi^{n+1})}{2}\big\vert \vu^{n+1} - \vu^n\big\vert^2
+\mathcal{R}_1^{n+1} + \mathcal{R}_2^{n+1},
\end{aligned}
\end{equation}
where $\mathcal{R}_1^{n+1}, \mathcal{R}_2^{n+1}$ are the two quadratic residual terms defining
\begin{equation}
\begin{aligned}
\mathcal{R}_1^{n+1}:=& b(\vu^{n+1}-\vu^n,\vu^{n+1}-\vu^n,\vvv^{n+1}),\\
\mathcal{R}_2^{n+1}:=&-\big(\alpha^\prime(\phi^n)\vu^n\cdot(\vvv^{n+1}-\vvv^{n}), \phi^{n+1}-\phi^n\big).
\end{aligned}
\end{equation}
\end{Lemma}
\begin{proof}
By the identity, it holds for all $a,b\in \mathbb{R}$
\begin{equation}\label{iden}
\frac{1}{2}(a^2-b^2) = (a-b)a-\frac{1}{2}\vert a-b\vert^2,
\end{equation}
yielding that
\begin{equation}\label{EnerStaDisS1}
\begin{aligned}
&J(\phi^{n+1},\vu^{n+1}) - J(\phi^{n},\vu^{n})\\
=&\frac{\mu}{2}\int_{\Omega} | {\rm D}\vu^{n+1}|^2 - |{\rm D}\vu^{n}|^2+ \int_{\Omega} \frac{\alpha(\phi^{n+1})}{2}|\vu^{n+1}|^2 - \frac{\alpha(\phi^{n})}{2}|\vu^{n}|^2 \\ 
=& a(\vu^{n+1},\vu^{n+1} - \vu^n) -\frac{\mu}{2}\int_{\Omega} |{\rm D}\vu^{n+1} - {\rm D}\vu^n|^2 \\
&+\int_{\Omega} \frac{\alpha(\phi^{n+1})}{2}(\big|\vu^{n+1}|^2-|\vu^n|^2 ) + \int_{\Omega} \frac{\alpha(\phi^{n+1})-\alpha(\phi^n)}{2}|\vu^n|^2.
\end{aligned}
\end{equation}
For last two terms in the last equation of (\ref{EnerStaDisS1}), we have
\begin{equation}\label{EnerStaDisS2}
\begin{aligned}
&\int_{\Omega} \frac{\alpha(\phi^{n+1})}{2}\big(|\vu^{n+1}|^2-|\vu^n|^2 \big) + \int_{\Omega} \frac{\alpha(\phi^{n+1})-\alpha(\phi^n)}{2}|\vu^n|^2\\
=&\big(\alpha(\phi^{n+1})\vu^{n+1}, \vu^{n+1}-\vu^n\big)- \int_{\Omega} \frac{\alpha(\phi^{n+1})}{2}\big\vert \vu^{n+1} - \vu^n\big\vert^2+\bigg(\frac{1}{2}\alpha^\prime(\phi^n)|\vu^n|^2,\phi^{n+1}-\phi^n \bigg).
\end{aligned}
\end{equation}
Given $(\phi^{n+1},\vu^{n+1},p^{n+1})$, take the test functions $(\vu^{n+1}-\vu^n, p^{n+1}-p^n)$ in (\ref{adjequ}) yielding that
\begin{equation}\label{EnerStaDisS3}\left\{
\begin{aligned}
&a(\vvv^{n+1},\vu^{n+1}-\vu^n)+b(\vu^{n+1}-\vu^n,\vu^{n+1},\vvv^{n+1}) + b(\vu^{n+1},\vu^{n+1}-\vu^n,\vvv^{n+1})\\
&+ (\alpha(\phi^{n+1}) \vvv^{n+1}, \vu^{n+1}-\vu^n)
-(q^{n+1}, \nabla \cdot (\vu^{n+1}-\vu^n)) \\
=&a(\vu^{n+1},\vu^{n+1}-\vu^n)+(\alpha(\phi^{n+1})\vu^{n+1}, \vu^{n+1}-\vu^n),\\
& (p^{n+1}-p^{n}, \nabla\cdot \vvv^{n+1}) =0.
\end{aligned}\right.
\end{equation}
Considering the Navier-Stokes equations (\ref{nsweak}) with two consecutive time steps, then the substraction leads to
\begin{equation}\label{EnerStaDisS4}
\left\{
\begin{aligned}
&a(\vu^{n+1}-\vu^n,\vvv^{n+1})+b(\vu^{n+1},\vu^{n+1},\vvv^{n+1})-b(\vu^{n},\vu^{n},\vvv^{n+1}) \\
&+ (\alpha(\phi^{n+1})\vu^{n+1},\vvv^{n+1})- (\alpha(\phi^{n})\vu^{n},\vvv^{n+1})-(p^{n+1}-p^{n}, \nabla \cdot\vvv^{n+1})= 0, \\
&(\nabla \cdot(\vu^{n+1}-\vu^n), q^{n+1})  = 0,
\end{aligned}\right.
\end{equation}
where we have taken the test functions by $(\vvv^{n+1}, q^{n+1})$. For the nonlinear terms in (\ref{EnerStaDisS4}), we have
\begin{equation}\label{EnerStaDisS5}
\begin{aligned}
&b(\vu^{n+1},\vu^{n+1},\vvv^{n+1})-b(\vu^{n},\vu^{n},\vvv^{n+1}) \\
=&b(\vu^{n+1},\vu^{n+1},\vvv^{n+1})-b(\vu^{n+1},\vu^{n},\vvv^{n+1})+b(\vu^{n+1},\vu^{n},\vvv^{n+1})-b(\vu^{n},\vu^{n},\vvv^{n+1}) \\
=&b(\vu^{n+1},\vu^{n+1}-\vu^n,\vvv^{n+1})+b(\vu^{n+1}-\vu^n,\vu^{n+1},\vvv^{n+1})-\mathcal{R}_1^{n+1},
\end{aligned}
\end{equation}
where the residual $\mathcal{R}_1^{n+1} = b(\vu^{n+1}-\vu^n,\vu^{n+1}-\vu^n,\vvv^{n+1})$ is second-order term of the substraction $\vu^{n+1}-\vu^{n}$. Similarly, we have
\begin{equation}\label{EnerStaDisS6}
\begin{aligned}
&(\alpha(\phi^{n+1})\vu^{n+1},\vvv^{n+1})- (\alpha(\phi^{n})\vu^{n},\vvv^{n+1})\\
=&(\alpha(\phi^{n+1})\vu^{n+1},\vvv^{n+1})- (\alpha(\phi^{n+1})\vu^{n},\vvv^{n+1})+(\alpha(\phi^{n+1})\vu^{n},\vvv^{n+1})-(\alpha(\phi^{n})\vu^{n},\vvv^{n+1})\\
=&\big(\alpha(\phi^{n+1})(\vu^{n+1}-\vu^n),\vvv^{n+1}\big)+\big(\alpha^\prime(\phi^n)\vu^n\vvv^{n+1}, \phi^{n+1}-\phi^n\big)\\
=&\big(\alpha(\phi^{n+1})(\vu^{n+1}-\vu^n),\vvv^{n+1}\big)+\big(\alpha^\prime(\phi^n)\vu^n\vvv^{n}, \phi^{n+1}-\phi^n\big)-\mathcal{R}_2,
\end{aligned}
\end{equation}
where the residual is $\mathcal{R}_2^{n+1}=-\big(\alpha^\prime(\phi^n)\vu^n(\vvv^{n+1}-\vvv^{n}), \phi^{n+1}-\phi^n\big)$. Thus from (\ref{EnerStaDisS3}) to (\ref{EnerStaDisS6}), we obtain
\begin{equation}\label{EnerStaDisS7}
\begin{aligned}
&a(\vu^{n+1},\vu^{n+1}-\vu^n)+(\alpha(\phi^{n+1})\vu^{n+1}, \vu^{n+1}-\vu^n)\\
=& -\big(\alpha^\prime(\phi^n)\vu^n\cdot \vvv^{n}, \phi^{n+1}-\phi^n\big)+\mathcal{R}_1^{n+1} + \mathcal{R}_2^{n+1}.
\end{aligned}
\end{equation}
Combining (\ref{EnerStaDisS1}), (\ref{EnerStaDisS2}) and (\ref{EnerStaDisS7}), we conclude the result.
\end{proof}

We are prepared to deduce the property of unconditional energy stable with semi-implicit gradient flow scheme. 
\begin{Theorem}\label{EnerDecThm}
Let $S_0, S_1$ are positive values independent of $\phi$ such that
\begin{equation}\left\{
\begin{aligned}
&\frac{1}{2\epsilon_2}\|\omega^{(2)}(\phi)\|_{L^\infty(\Omega)} +K_{\Omega}\tilde{C}_1^2\mathcal{C}_2+\frac{1}{2}\alpha_0(1+C_p)^2\mathcal{C}_1\tilde{C}_2+\frac{\beta\vert \Omega\vert^2}{2}\leq S_0,\\
&\frac{1}{2}\alpha_0(1+C_p)^2\mathcal{C}_1\tilde{C}_2\leq S_1.
\end{aligned}\right.
\end{equation}
The unconditional energy stable holds for the stabilized gradient flow scheme (\ref{GradFlowScheme})
\begin{equation}
    \mathcal{W}(\phi^{n+1},\vu^{n+1}) - \mathcal{W}(\phi^{n},\vu^{n}) \leq  \tau (\nu^{n+1},\mathcal{G}\nu^{n+1}).
\end{equation}
\end{Theorem}
\begin{proof}
The residue terms in (\ref{subJinLEMMA}) can be bounded by Lemma \ref{boundLemma}
\begin{equation}
\begin{aligned}
\vert \mathcal{R}_1^{n+1}\vert&=\vert b(\vu^{n+1}-\vu^n,\vu^{n+1}-\vu^n,\vvv^{n+1})\vert\\
&\leq K_{\Omega} \|\nabla(\vu^{n+1}-\vu^n)\|_{\textbf{L}^2(\Omega)}^2\| \nabla\vvv^{n+1}\|_{\textbf{L}^2(\Omega)}\\
&\leq K_{\Omega}\tilde{C}_1^2\mathcal{C}_2 \| \phi^{n+1}-\phi^n\|_{L^2(\Omega)}^2,
\end{aligned}
\end{equation}
and
\begin{equation}
\begin{aligned}
\vert \mathcal{R}_2^{n+1}\vert&=\vert \big(\alpha^\prime(\phi^n)\cdot\vu^n(\vvv^{n+1}-\vvv^{n}), \phi^{n+1}-\phi^n\big)\vert\\
&\leq \| \alpha^\prime(\phi^n)\vu^n\|_{\textbf{L}^4(\Omega)} \| \vvv^{n+1}-\vvv^{n}\|_{\textbf{L}^2(\Omega)}\| \phi^{n+1}-\phi^{n}\|_{{L}^4(\Omega)}\\
&\leq \alpha_0(1+C_p)^2 \|\nabla\vu^n\|_{\textbf{L}^2(\Omega)}\tilde{C}_2 \| \phi^{n+1}-\phi^{n}\|_{{L}^2(\Omega)}\|\nabla( \phi^{n+1}-\phi^{n})\|_{\textbf{L}^2(\Omega)}\\
&\leq \frac{1}{2}\alpha_0(1+C_p)^2\mathcal{C}_1\tilde{C}_2(\| \phi^{n+1}-\phi^{n}\|_{{L}^2(\Omega)}^2+\|\nabla( \phi^{n+1}-\phi^{n})\|_{\textbf{L}^2(\Omega)}^2).
\end{aligned}
\end{equation}
The subtraction between adjacent Ginzburg-Landau energy gives
\begin{equation}\label{subGinzLan}
 \begin{aligned}
 &\int_{\Omega} \bigg[\frac{\epsilon_1}{2} \big| \nabla \phi^{n+1} \big |^2+\frac{1}{\epsilon_2}  \omega(\phi^{n+1})\bigg] \dx - \int_{\Omega} 
 \bigg[ \frac{\epsilon_1}{2} \big | \nabla \phi^{n} \big |^2+\frac{1}{\epsilon_2}  \omega(\phi^{n})\bigg] \dx \\
 =& \int_{\Omega} \epsilon_1 \big(\nabla \phi^{n+1} - \nabla \phi^n \big)\cdot \nabla \phi^{n+1} \dx - \frac{1}{2}\int_{\Omega} \epsilon_1 \big|\nabla \phi^{n+1} - \nabla \phi^n \big|^2 \dx \\
 &+ \int_{\Omega} \frac{1}{\epsilon_2} \omega^\prime (\phi^n) (\phi^{n+1} - \phi^n)+ \frac{1}{2\epsilon_2}\omega^{(2)}(\zeta) (\phi^{n+1}-\phi^n)^2 \dx,
\end{aligned}
\end{equation}
where we use Taylor expansion up to the second order by taking $\zeta = \tau_0 \phi^{n+1} + (1-\tau_0 ) \phi^n, \tau_0 \in (0,1)$. Then use the linear property of the volume functional and apply the Cauchy-Schwarz inequality to obtain
\begin{equation}
\begin{aligned}
&\frac{\beta }{2} \big( V(\phi^{n+1})- \hat{V} \big)^2- \frac{\beta }{2}\big( V(\phi^{n})-\hat{V} \big)^2 \\
=& \frac{\beta}{2}\big(V(\phi^{n+1}) + V(\phi^n) - 2\hat{V}\big) (V(\phi^{n+1}) - V(\phi^n)) \\
=& \frac{\beta}{2}\big(2V(\phi^n) - 2\hat{V} + V(\phi^{n+1}) - V(\phi^n)\big) (V^\prime(\phi^n), \phi^{n+1}- \phi^n)\\
=& \beta\big(V(\phi^n)-\hat{V}\big)\int_{\Omega} V^\prime(\phi^n) (\phi^{n+1}-\phi^n)\dx + \frac{\beta}{2}\bigg[\int_{\Omega} V^\prime(\phi^n) (\phi^{n+1}-\phi^n)\dx\bigg]^2\\
\leq& \beta\big(V(\phi^n)-\hat{V}\big)\int_{\Omega} V^\prime(\phi^n) (\phi^{n+1}-\phi^n)\dx+ \frac{\beta}{2}\| V^\prime(\phi^n)\|_{L^2(\Omega)}^2 \| \phi^{n+1}-\phi^n\|_{L^2(\Omega)}^2.
\end{aligned}
\end{equation}
We conduct the estimation via Lemma \ref{StaJ} and (\ref{subGinzLan})
\begin{equation}
\begin{aligned}
&\mathcal{W}(\phi^{n+1},\vu^{n+1}) - \mathcal{W}(\phi^{n},\vu^{n})\\
=&-\frac{\mu}{2}\int_{\Omega} \vert{\rm D}\vu^{n+1} - {\rm D}\vu^n\vert^2- \int_{\Omega} \frac{\alpha(\phi^{n+1})}{2}\big\vert \vu^{n+1} - \vu^n\big\vert^2\\
&+\frac{1}{2}\big(\alpha^\prime(\phi^n)(\vu^n)^2-\alpha^\prime(\phi^n)\vu^n\cdot\vvv^{n},\phi^{n+1}-\phi^n \big)\\
&+\epsilon_1( \nabla\phi^{n+1}-\nabla\phi^{n},\nabla\phi^{n+1})-\frac{\epsilon_1}{2}\Vert \nabla\phi^{n+1}-\nabla\phi^{n}\Vert_{\textbf{L}^2(\Omega)}^2+\mathcal{R}_1^{n+1} + \mathcal{R}_2^{n+1}\\
&+\frac{1}{\epsilon_2}(\omega^\prime(\phi^n),\phi^{n+1}-\phi^n)+\frac{1}{2\epsilon_2}(\omega^{(2)}(\zeta_1),(\phi^{n+1}-\phi^n)^2)\\
&+\beta(V(\phi^{n})-\hat{V})(V^\prime(\phi^{n}),\phi_{n+1}-\phi_n)+ \frac{\beta}{2}\bigg[\int_{\Omega} V^\prime(\phi^n) (\phi^{n+1}-\phi^n)\dx\bigg]^2,
\end{aligned}
\end{equation}
and
\begin{equation}
\begin{aligned}
&\mathcal{W}(\phi^{n+1},\vu^{n+1}) - \mathcal{W}(\phi^{n},\vu^{n})\\
\leq &\bigg(\frac{1}{2}\alpha^\prime(\phi^n)|\vu^n|^2-\alpha^\prime(\phi^n)\vu^n\cdot\vvv^{n},\phi^{n+1}-\phi^n \bigg)+\mathcal{R}_1^{n+1} + \mathcal{R}_2^{n+1}\\
&+(\nabla\phi^{n+1}-\nabla\phi^{n},\epsilon_1\nabla\phi^{n+1})
+\frac{1}{\epsilon_2}(\omega^\prime(\phi^n),\phi^{n+1}-\phi^n)+\frac{1}{2\epsilon_2}(\omega^{(2)}(\zeta_1),(\phi^{n+1}-\phi^n)^2)\\
&+\beta(V(\phi^{n})-\hat{V})(V^\prime(\phi^{n}),\phi_{n+1}-\phi_n)+ \frac{\beta}{2}\bigg[\int_{\Omega} V^\prime(\phi^n) (\phi^{n+1}-\phi^n)\dx\bigg]^2\\
=&\tau (\nu^{n+1},\mathcal{G}\nu^{n+1})-(S_0,(\phi^{n+1}-\phi^n)^2)-(S_1,|\nabla\phi^{n+1}-\nabla\phi^n|^2)+\mathcal{R}_1^{n+1} + \mathcal{R}_2^{n+1} \\
&+\frac{1}{2\epsilon_2}(\omega^{(2)}(\zeta_1),(\phi^{n+1}-\phi^n)^2)+ \frac{\beta}{2}\bigg[\int_{\Omega} V^\prime(\phi^n) (\phi^{n+1}-\phi^n)\dx\bigg]^2\\
\leq&\bigg(\frac{1}{2\epsilon_2}\|\omega^{(2)}(\phi)\|_{L^\infty(\Omega)} +K_{\Omega}\tilde{C}_1^2\mathcal{C}_2+\frac{1}{2}\alpha_0(1+C_p)^2\mathcal{C}_1\tilde{C}_2+\frac{\beta\vert \Omega\vert^2}{2}-S_0\bigg) \|\phi^{n+1}-\phi^n\|_{L^2(\Omega)}^2\\
&+\bigg(\frac{1}{2}\alpha_0(1+C_p)^2\mathcal{C}_1\tilde{C}_2-S_1\bigg) \|\nabla\phi^{n+1}-\nabla\phi^n\|_{L^2(\Omega)}^2+ \tau(\nu^{n+1},\mathcal{G}\nu^{n+1})\\
\leq& \tau (\nu^{n+1},\mathcal{G}\nu^{n+1}),
\end{aligned}
\end{equation}
thanks to the nonnegative operator $\mathcal{G}$ and sufficient large values $S_0$ and $S_1$. 
\end{proof}
\begin{Remark}
The popular method to construct an energy stable scheme of gradient flow is the class of convex splitting method \cite{Eyre1998} involving inner iteration. While it requires updating P.D.E. repeatedly for each time step in the case of topology optimization. The stabilization treats the nonlinear terms explicitly and adds a stabilization term to avoid strict time step constraints \cite{Shen1999}. In our case, the instability factors come from the nonlinear term in Navier-Stokes equations and the introduction of adjoint variables. The other method to construct the energy stable scheme is introducing a scalar auxiliary variable (see \cite{Shen2019}). However, such a method can only keep the energy stability of the modified energy instead of the original energy. 
\end{Remark}

Then we introduce the bounded value function space
\begin{equation}
\textbf{BV}(\Omega; [0,1]):=\{ \xi\in L^\infty(\Omega),\ 0\leq \xi \leq 1 \},
\end{equation}
and the corresponding projection operator $\mathcal{P}$ defines
\begin{equation}\label{proj}
\begin{aligned}
\mathcal{P}(\cdot): L^\infty(\Omega) &\rightarrow L^\infty(\Omega),\\
\zeta &\longmapsto \inf_{\xi\in \textbf{BV}(\Omega; [0,1])} \Vert \xi - \zeta\Vert_{L^\infty(\Omega)}.
\end{aligned}
\end{equation}
Modify the volume function $V(\phi):=\int_{\Omega}(1-\mathcal{P}(\phi))$ and $\alpha(\phi)=\alpha_0(1-\mathcal{P}(\phi))$ to obtain the following result.

\begin{Lemma}\label{prolemma}
Let $\tilde{\phi} = \mathcal{P}(\phi)$ be the projected phase field function where the projection operator is defined (\ref{proj}). Suppose $({\phi}, \vu, p )$ is the solution pair of (\ref{nsweak}). Then $(\tilde{\phi}, \vu, p)$ is still the solution pair of (\ref{nsweak}) and the total energy decreases satisfying that
\begin{equation}
\mathcal{W}(\tilde{\phi},\vu)\leq \mathcal{W}(\phi,\vu).
\end{equation}
\end{Lemma}
\begin{proof}
The projection operator satisfies $\mathcal{P}^2=\mathcal{P}$. Then using the definition of the permeability function defined on the whole domain and the volume function, it holds
\begin{equation}\label{invEV}
\begin{aligned}
&\alpha(\tilde{\phi})=\alpha(\mathcal{P}({\phi}))=\alpha_0(1-\mathcal{P}^2(\phi))= \alpha_0(1-\mathcal{P}(\phi)) = \alpha(\phi),\\
&V(\tilde{\phi})=\int_{\Omega}[1- \mathcal{P}^2(\phi)]\dx =\int_{\Omega} 1-\mathcal{P}(\phi)\dx  = V(\phi),
\end{aligned}
\end{equation}
such that $(\tilde{\phi}, \vu)$ is the solution pair of (\ref{nsweak}). Let ${\Omega}$ be partitioned into two disjoint domains $\overline{{\Omega}}= \overline{D_1} \cup \overline{D_2}$ where
\begin{equation}
D_1 := \{\bm x\in \Omega,\ \phi\in [0,1]\}, \ D_2 := \{\bm x\in \Omega,\ \phi\notin [0,1]\}.
\end{equation}
Then we obtain
\begin{equation}\label{pros1}
\begin{aligned}
\int_{\Omega} \frac{\epsilon_1}{2}|\nabla \phi|^2 \dx & = \int_{D_1} \frac{\epsilon_1}{2}|\nabla \phi|^2 \dx+\int_{D_2} \frac{\epsilon_1}{2}|\nabla \phi|^2 \dx\\
&\geq \int_{D_1} \frac{\epsilon_1}{2}|\nabla {\phi}|^2 \dx= \int_{\Omega} \frac{\epsilon_1}{2}|\nabla \tilde{\phi}|^2 \dx,
\end{aligned}
\end{equation}
where $\nabla\tilde{\phi}$ vanishes on $D_2$ for $\tilde{\phi}$ is a.e. constant. Furthermore, by the nonnegativity of the double well potential, we have
\begin{equation}\label{pros2}
\begin{aligned}
\int_{\Omega}\frac{1}{\epsilon_2} \omega(\phi)\dx&=\int_{D_1}\frac{1}{\epsilon_2} \omega(\phi)\dx+\int_{D_2}\frac{1}{\epsilon_2} \omega(\phi)\dx \\
&\geq  \int_{D_1}\frac{1}{\epsilon_2} \omega({\phi})\dx= \int_{\Omega}\frac{1}{\epsilon_2} \omega(\tilde{\phi})\dx,
\end{aligned}
\end{equation}
where $\omega(\tilde{\phi})$ vanishes on $D_2$ for $\tilde{\phi}$ takes value on 0 or 1. The invariant of the permeability function and volume function under the projection leads to
\begin{equation}\label{pros3}
J(\tilde{\phi}, \vu(\tilde{\phi})) = J(\phi, \vu(\phi)),
\end{equation}
and
\begin{equation}\label{pros4}
\frac{1}{2}\beta\big(V(\tilde{\phi})-\hat{V}\big)^2=\frac{1}{2}\beta\big(V({\phi})-\hat{V}\big)^2.
\end{equation}
Combining (\ref{pros1}) (\ref{pros2}) (\ref{pros3}) and (\ref{pros4}), we conclude the estimate.
\end{proof}

\begin{remark}
Given the previous phase field function $\phi^{n-1}$, then compute the velocity field $\vu^{n-1}$ via the Navier-Stokes equations. After that update the phase field function $\phi^n$ via the stabilized gradient flow (\ref{GradFlowScheme}) and then compute the corresponding velocity field $\vu^n$. In this way, the minimizing sequence $\{\phi^n, \vu^n \}_{n=1}^\infty$ is generated by repeating the above procedure. From Theorem \ref{EnerDecThm}, the sequence $\{\phi^n, \vu^n \}_{n=0}^\infty$ satisfies the energy dissipation $$\mathcal{W}(\phi^{n+1}, \vu^{n+1})\leq \mathcal{W}(\phi^{n}, \vu^{n})\leq\cdots\leq \mathcal{W}(\phi^{0}, \vu^{0}),$$ which guarantees the convergence and monotonicity. However, the phase field function may exceed the range $[0,1]$ causing inaccuracy in computing the Navier-Stokes equations. 

Another way to construct the convergent sequence meanwhile keeping phase field function in the range $[0,1]$ is computing $(\phi^n, \vu^n, \vvv^n)$ by (\ref{GradFlowScheme}), (\ref{nsweak}) and (\ref{adjequ}) first. Then use the projection $\mathcal{P}(\cdot)$ on phase field function to obtain $(\tilde{\phi}^n, \vu^n)$. We note that $(\tilde{\phi}^n, \vu^n)$ is still the solution pair of the Navier-Stokes equations (\ref{nsweak}) because the permeability function does not change by the projection operator. Then, use the stabilized gradient flow to obtain the next decreasing iteration
$$\mathcal{W}(\phi^{n+1},\vu^{n+1})\leq \mathcal{W}(\tilde{\phi}^{n},\vu^{n}).$$
Finally, we alternatively take the projection step and the stabilized gradient flow step to generate the sequence $\{\tilde{\phi}^n, \phi^n, \vu^n \}_{n=1}^\infty$ such that $$\mathcal{W}(\tilde{\phi}^{n+1},\vu^{n+1})\leq \mathcal{W}(\phi^{n+1},\vu^{n+1})\leq \mathcal{W}(\tilde{\phi}^{n},\vu^{n})\leq \mathcal{W}({\phi^{n}},\vu^{n})\leq\cdots\leq \mathcal{W}(\tilde{\phi}^{0}, \vu^{0})\leq \mathcal{W}(\phi^{0}, \vu^{0}),$$
which preserves the energy dissipation and bounds the phase field function in $[0,1]$.
\end{remark}

\section{Numerical realization}
In this section, we introduce the finite element to discretize the Navier-Stokes equations (\ref{nsweak}), the adjoint equation (\ref{adjequ}) and the gradient flow (\ref{GradFlowScheme}) of the phase field. 
\subsection{Finite element discretization}
Consider a family of unstructured meshes $\{ \mathcal{T} _h\}_{h >0}$ satisfying the union of triangular units $\overline{\Omega} = \bigcup _{K \in\mathcal{T}_h} \overline{K}$, where the mesh size is $ h:= \max_{K\in\mathcal{T}_h} h_K$  with  $h_K$ being the diameter of any $K \in \mathcal{T}_h$. 
Let us consider the conforming finite element subspaces characterized as the discrete velocity function space $\textbf{U}_h\subset \textbf{H}^1(\Omega)$ and the discrete pressure function as well as the phase field function space $W_h\subset L^2(\Omega)$. We shall also admit a \emph{compatibility} condition \cite{Brezzi,GR,Temam} between the discrete velocity and pressure spaces $\mathbf{U}_h$ and $W_h$ by assuming that there exists a positive constant $\beta_0>0$ such that
\begin{equation}\label{LBB}
\inf _{q_h \in W_h} \sup _{\mathbf{v}_h \in \mathbf{U}_h} \frac{ (q_h, \nabla\cdot \mathbf{v}_h )}{\left\|\nabla \mathbf{v}_h\right\|_{\textbf{L}^2(\Omega)}\left\|q_h\right\|_{L^2(\Omega)}} \geqslant \beta_0.
\end{equation}
Denote the space of bubble functions by
$\textbf{P}:=[\mathbb{P}_1\oplus {\rm span}\, b_K]^d$,
where the bubble function
$b_K:=\Pi_{i=1}^{d+1}\lambda_i$ and $\lambda_i(i=1,\cdots,d+1)$ are the barycentric coordinates of $K$. Then the MINI ($\mathbb{P}_1$-bubble/$\mathbb{P}_1$) element \cite{Brezzi,GR} for discretization of Navier-Stokes equations is given by
\begin{eqnarray*}
	&&\textbf{U}_h=\{\vec{v}_h\in C^0(\Omega)^d \,\big| \,\vec{v}_h|_K\in \textbf{P}\ \forall\ K\in\mathcal{T}_h\},\\
	&&W_h=\{q_h\in  C^0(\Omega) \, \big\lvert \,q_h|_K\in \mathbb{P}_1\ \forall\ K\in\mathcal{T}_h\}.
\end{eqnarray*}
The following subspaces are necessary to describe the Dirichlet boundary condition for the velocity field $\textbf{U}^h_{d} = \textbf{U}_h \cap \textbf{H}^1_d (\Omega)$ and for the adjoint function $\textbf{U}^h_{0} = \textbf{U}_h \cap \textbf{H}^1_{00} (\Omega)$. The phase field function is discretized in piece-wise linear function space $\phi_h \in W_h$. The discrete variational problem of Navier-Stokes equations (\ref{nsweak}) reads: find $(\vu_h, p_h)\in \textbf{U}^h_{d}\times W_h$ such that 
\begin{equation}\label{disNS}
\left\{
\begin{aligned}
&a(\vu_h,\vvv_h)+b(\vu_h,\vu_h,\vvv_h) + (\alpha(\phi_h)\vu_h,\vvv_h)-(p_h, \nabla \cdot\vvv_h)= (\vf,\vvv_h) \quad  && \forall \vvv_h \in \textbf{U}_h, \\
&(\nabla \cdot\vu_h,q_h)  = 0\quad && \forall q_h \in W_h.
\end{aligned}\right.
\end{equation}
Similarly, the discrete problem of adjoint equations (\ref{adjequ}) reads: find $(\vvv_h, q_h)\in \textbf{U}^h_{00}\times W_h$ such that 
\begin{equation}\left\{
\begin{aligned}
&a(\vvv_h,\vw_h)+\hat{b}(\vu_h,\vvv_h,\vw_h) - b(\vu_h,\vvv_h,\vw_h)+ (\alpha(\phi_h) \vvv_h, \vw_h)-(q_h, \nabla \cdot \vw_h)\\
&=a(\vu_h,\vw_h)+(\alpha(\phi_h)\vu_h, \vw_h),\qquad \forall \vw_h \in \textbf{U}_h,\\
& (z_h, \nabla\cdot \vvv_h) =0,\qquad  \forall  z_h \in W_h.
\end{aligned}\right.
\end{equation}
The discrete variational problem of gradient flow (\ref{GradFlowScheme}) is given by: find $(\phi^{n+1}_h, \nu_h^{n+1}) \in W_h\times W_h$ such that
\begin{equation}\label{disGradFlow}\left\{
\begin{aligned}
&\frac{1}{\tau}(\phi^{n+1}_h-\phi^{n}_h, \zeta_h) = (\mathcal{G}_h \nu^{n+1}_h, \zeta_h),\qquad && \forall  \zeta_h\in W_h \\
&(\nu^{n+1}_h, \psi_h) = (\epsilon_1 \nabla\phi^{n+1}_h, \nabla \psi_h) + \big(\mathcal{U}(\phi^n_h, \vu^n_h, \vvv^n_h),\psi_h\big) + \big(\mathcal{S}(\phi^{n+1}_h-\phi^n_h), \psi_h\big), && \forall \psi_h \in W_h,
\end{aligned}\right.
\end{equation}
where $\mathcal{G}_h$ is the discrete nonpositive symmetric operator such as the discrete Laplacian $\mathcal{G}_h = -\Delta_h$ for $H^{-1}$ gradient flow and the identity operator for $L^2$ gradient flow $\mathcal{G}_h = - {\rm I}$. For Allen-Cahn gradient flow, take the test function $\psi_h = \zeta_h$ in the second equation of (\ref{disGradFlow}) then insert it into the first equation to simplify the expression. Hence no intermediate variable is introduced. Since the stability proofs of stabilized semi-implicit schemes (\ref{GradFlowScheme}) are all variational, they can be directly extended to fully discrete stabilized semi-implicit schemes with mixed finite element methods.
\begin{Theorem}
Let assumptions in Theorem \ref{EnerDecThm} hold. Consider $(\phi^n_h, \vu^n_h, p^n_h)$ and $(\phi^{n+1}_h, \vu^{n+1}_h, p^{n+1}_h)$ are the solution pairs of (\ref{disNS}) where $\phi^{n+1}_h$ is updated by (\ref{disGradFlow}). Then the total energy has the monotonic-decaying property:
\begin{equation}
\mathcal{W} (\phi^{n+1}_h,\vu^{n+1}_h) - \mathcal{W} (\phi^{n}_h,\vu^{n}_h) \leq \tau(\nu^{n+1}_h,\mathcal{G}_h \nu^{n+1}_h),
\end{equation}
where $S_0$ and $S_1$ are sufficiently large numbers independent of $\phi_h$.
\end{Theorem}
\begin{proof}
Replace the continuous variables $(\phi^j, \vu^j, p^j)$ with the discrete variables $(\phi^j_h, \vu^j_h, p^j_h )$ as well as for the adjoint variables ($j=n,n+1$). Then follow the procedure of Theorem \ref{EnerDecThm} to obtain the conclusion.
\end{proof}

To verify the monotonic-decaying property of the projection in the fully-discrete sense, we refer to \cite[Lemma 4.1]{LiFu} with the phase field function discretized by the piecewise linear finite element method
\begin{equation}\label{disProDecay}
  \|\nabla(\mathcal{P}_h \phi_h) \|_{L^2(\Omega)}\leq \|\nabla \phi_h \|_{L^2(\Omega)},
\end{equation}
where $\mathcal{P}_h$ defines the corresponding discrete projection operator
\begin{equation}\label{projdis}
\begin{aligned}
\mathcal{P}_h(\cdot): W_h &\rightarrow W_h,\\
\phi_h &\longmapsto \inf_{\tilde{\phi}_h \in \textbf{BV}(\Omega; [0,1])\cap W_h} \Vert \tilde{\phi}_h - \phi_h \Vert_{L^\infty(\Omega)}.
\end{aligned}
\end{equation}
Denote the discrete permeability $\alpha(\phi_h):=\alpha_0(1-\mathcal{P}_h(\phi_h))$. Then the following monitonicity property holds.
\begin{Lemma}
Let $\tilde{\phi}_h = \mathcal{P}_h(\phi_h)$ be the projected phase field function where the projection operator is defined (\ref{disNS}). Suppose that $({\phi}_h, \vu_h, p_h)$ is the solution pair of (\ref{nsweak}), then $(\tilde{\phi}_h, \vu_h, p_h)$ is still the solution pair of (\ref{disNS}) and the total energy decreases satisfying that
\begin{equation}
\mathcal{W}(\tilde{\phi}_h,\vu_h)\leq \mathcal{W}(\phi_h,\vu_h).
\end{equation}
\end{Lemma}
\begin{proof}
Follow the Lemma \ref{prolemma} and combine (\ref{disProDecay}) to conclude the proof.
\end{proof}

The Lagrange multiplier $\ell^n$ is introduced to further eliminate the volume error for the Allen-Cahn gradient flow. Furthermore, the nonlinear term can be expressed by
\begin{equation}
\mathcal{U}(\phi^n_h, \vu^n_h, \vvv^n_h):=\frac{1}{\epsilon_2} \omega^\prime(\phi^{n}_h) + \frac{\eta_1}{\eta_2}\bigg( j_\phi(\phi^{n}_h,\vu^{n}_h)-\alpha^\prime(\phi^{n}_h)\vu^{n}_h\cdot\vvv^{n}_h\bigg) + \beta(V(\phi^{n}_h)-\hat{V})V^\prime(\phi^{n}_h)+\ell^n,
\end{equation}
where $\eta_1$ is the weighted parameter, $\eta_2:=\| j_\phi(\phi^{n}_h,\vu^{n}_h)-\alpha^\prime(\phi^{n}_h)\vu^{n}_h\cdot\vvv^{n}_h\|_{L^2(\Omega)}$ is the normalized factor. For updating the Lagrange multiplier, a Uzawa type scheme reads
\begin{equation}\label{Uzawa}
    \ell^{n+1} = \ell^n + \beta ( V(\phi_h^n)-V_0).
\end{equation}
Note that the discrete Navier-Stokes system in (\ref{disNS}) has a nonlinear convection term. A typically efficient Newton scheme preserving locally quadratic convergence rate is proposed by solving numerically a series of Oseen problems: Given $(\vu_h,p_h)$ for the previous approximation, find $({\vu_h^{\Delta}}, p_h^{\Delta})\in \textbf{U}^h_{d}\times W_h$ such that $ \forall\, (\vvv_h,q_h) \in \textbf{U}_h\times  W_h$
\begin{equation}
\left\{
\begin{aligned}
&a(\vu_h^{\Delta},\vvv_h)+b(\vu_h^{\Delta},\vu_h,\vvv_h)+b(\vu_h,\vu_h^{\Delta},\vvv_h) + (\alpha(\phi_h)\vu_h^{\Delta},\vvv_h)\\& -(p_h^{\Delta}, \nabla \cdot\vvv_h)= (\vf,\vvv_h)-a(\vu_h,\vvv_h)-b(\vu_h,\vu_h,\vvv_h)- (\alpha(\phi_h)\vu_h,\vvv_h)+(p_h, \nabla \cdot\vvv_h), \\
&(\nabla \cdot\vu_h^{\Delta},q_h)  = 0.
\end{aligned}\right.
\end{equation}
Then the approximate solution pair is updated by $(\vu_h, p_h) \leftarrow (\vu_h, p_h) + (\vu_h^{\Delta}, p_h^{\Delta}).$ In this section, we propose a topology optimization algorithm based on the scheme (\ref{disGradFlow}). We use the trick that once the state and adjoint variables are solved, the phase field is evolved via the gradient flow scheme for several steps to further improve the efficiency. Now, we are prepared to present Algorithm \ref{algAC} (Allen-Cahn) and Algorithm \ref{algCH} (Cahn-Hilliard) for topology optimization using the semi-implicit gradient flow scheme. We note that the projection $\mathcal{P}(\cdot)$ can not be used in the Cahn-Hilliard gradient for it may break the mass conservation.
\begin{algorithm}[htbp]
\SetAlgoLined
    \caption{Projected Allen-Cahn gradient flow algorithm for topology optimization of Navier-Stokes flows}\label{algAC}
    \KwData{Given the maximum iteration times $N, N_{\phi}$, the stopping tolerance $\epsilon_{\vu}$ \\ Initialize phase field function $\phi_0$ and $n=0$}
    \While{$n \leq N$}{
    \emph{Step 1}: Solve state variables $(\bm u^{n+1}_h, p_h^{n+1})$ of Navier-Stokes equations \\
        \While{$\Vert \vu_h^\Delta \Vert_{\textbf{L}^\infty (\Omega)} \leq \epsilon_{\vu} $ }
        {(1). Compute the increment $ (\vu_h^\Delta, p_h^\Delta)$ via Newton iteration\\
        (2). Update the solution by 
        $\vu_h \leftarrow \vu_h +  \vu_h^\Delta,\quad p_h \leftarrow p_h +  p_h^\Delta$\\
        }
        \emph{Step 2}: Solve adjoint variables $(\bm v^{n+1}_h, q^{n+1}_h)$ from adjoint equations\\
        \quad Set $\phi^{n+1}_0 \leftarrow \phi^n$ and $k =0$\\
        \emph{Step 3}: Update the phase field function via Allen-Cahn gradient flow\\
        \While{$k \leq N_{\phi}$ }{
            (1). Update the phase field function $\phi^{n+1}_{k+1}$ with semi-implicit scheme \\
            (2). Use the projection by $\phi^{n+1}_{k+1}  \leftarrow \mathcal{P}(\phi^{n+1}_{k+1}) $\\}      
Set $\phi^{n+1}  \leftarrow \phi^{n+1}_{N_{\phi}}, \, n \leftarrow n+1$\\
    \emph{Step 4}: Update the Lagrange multiplier}
\end{algorithm}

\begin{algorithm}[htbp]
\SetAlgoLined
    \caption{Cahn-Hilliard gradient flow algorithm for topology optimization of Navier-Stokes flows}\label{algCH}
    \KwData{Given the maximum outer and inner iteration numbers $N, N_{\phi}$, and stopping tolerance $\epsilon_{\vu}$ \\ Initialize phase field function $\phi_0$ and $n=0$}
    \While{$n \leq N$}
    {\emph{Step 1}: Solve state variables $(\bm u^{n+1}_h, p_h^{n+1})$ of Navier-Stokes equations \\
    \emph{Step 2}: Solve adjoint variables $(\bm v^{n+1}_h, q^{n+1}_h)$ from adjoint equations\\
    \quad Set $\phi^{n+1}_0 \leftarrow \phi^n$ and $k =0$\\
    \emph{Step 3}: Update the phase field function via Cahn-Hilliard gradient flow \\
        \While{$k \leq N_{\phi}$ }
            {Update the phase field function $\phi^{n+1}_{k+1}$ with semi-implicit scheme \\
            }
    Set $\phi^{n+1}  \leftarrow \phi^{n+1}_{N_{\phi}}, \, n \leftarrow n+1$}
\end{algorithm}
\subsection{Numerical experiments}
Numerical simulations are performed with FreeFem++ \cite{Hecht}. All numerical results are performed on a computer with 12th Gen Intel(R) Core(TM) i7-12700 2.10 GHz and 16 GB memory. 
Set $\alpha(\phi):=\alpha_0(1-\mathcal{P}(\phi))$ where the permeability coefficient $\alpha_0=1000$ for the following Examples.
\begin{figure}[htbp]
\centering
\begin{tikzpicture}[scale=3.2]
	\draw[thick,fill=gray!10!white] (0.0,0.0) -- (1.0,0.0) -- (1.0,1.0) -- (0.0,1.0) -- cycle;
	\draw [thick, arrows = {-Latex[width=3pt, length=6pt]}] (0.0,0.1) -- (0.15,0.1);
	\draw [thick, arrows = {-Latex[width=3pt, length=6pt]}] (0.0,0.2) -- (0.15,0.2);
	\draw [thick, arrows = {-Latex[width=3pt, length=6pt]}] (0.0,0.3) -- (0.15,0.3);
	\draw [thick, arrows = {-Latex[width=3pt, length=6pt]}] (0.0,0.4) -- (0.15,0.4);
	\draw [thick, arrows = {-Latex[width=3pt, length=6pt]}] (0.0,0.5) -- (0.15,0.5);
	\draw [thick, arrows = {-Latex[width=3pt, length=6pt]}] (0.0,0.6) -- (0.15,0.6);
	\draw [thick, arrows = {-Latex[width=3pt, length=6pt]}] (0.0,0.7) -- (0.15,0.7);
	\draw [thick, arrows = {-Latex[width=3pt, length=6pt]}] (0.0,0.8) -- (0.15,0.8);
	\draw [thick, arrows = {-Latex[width=3pt, length=6pt]}] (0.0,0.9) -- (0.15,0.9);
	\draw [thick, arrows = {-Latex[width=3pt, length=6pt]}] (1.0,1/3) -- (1.15,1/3);
	\draw [thick, arrows = {-Latex[width=3pt, length=6pt]}] (1.0,2/3) -- (1.15,2/3);
	\draw [thick, arrows = {-Latex[width=3pt, length=6pt]}] (1.0,0.5) -- (1.15,0.5);
	\draw [thick, arrows = {-Latex[width=3pt, length=6pt]}] (1.0,0.5-1/12) -- (1.15,0.5-1/12);
	\draw [thick, arrows = {-Latex[width=3pt, length=6pt]}] (1.0,0.5+1/12) -- (1.15,0.5+1/12);
	\draw [arrows = {-Latex[width=5pt, length=5pt]}] (-0.05,0.0) -- (-0.05,1.0);
	\draw [arrows = {-Latex[width=5pt, length=5pt]}]   (-0.05,1.0)--(-0.05,0.0);
	\draw [arrows = {-Latex[width=5pt, length=5pt]}]   (0,-0.05)--(1,-0.05);
	\draw [arrows = {-Latex[width=5pt, length=5pt]}]   (1,-0.05) -- (0,-0.05);
	\draw [arrows = {-Latex[width=5pt, length=5pt]}] (1-0.05,1/3) -- (1-0.05,2/3);
	\draw [arrows = {-Latex[width=5pt, length=5pt]}]   (1-0.05,2/3)--(1-0.05,1/3);
	\draw (0.5,-0.1) node[scale=0.7] {$1L$};
	\draw (-0.1,0.5) node[scale=0.7] {$1L$};
	\draw (0.9,0.5) node[scale=0.7] {$\frac{1}{3}L$};
	\draw (0.5,0.5) node[scale=1] {$\Omega$};
\end{tikzpicture}
\qquad
\begin{tikzpicture}[scale=3.2]
\draw[thick,fill=gray!10!white] (0.0,0.0) -- (1.5,0.0) -- (1.5,1.0) -- (0.0,1.0) -- cycle;
\draw[thick, dashed] (0.0,0.15) .. controls (0.15,0.25) .. (0.0,0.35);
\draw[thick, dashed] (0.0,0.65) .. controls (0.15,0.75) .. (0.0,0.85);
\draw [arrows = {-Latex[width=4pt, length=4pt]}]   (0.0,0.25) -- (0.11,0.25);
\draw [arrows = {-Latex[width=4pt, length=4pt]}]   (0.0,0.3) -- (0.075,0.3);
\draw [arrows = {-Latex[width=4pt, length=4pt]}]   (0.0,0.2) -- (0.075,0.2);
\draw [arrows = {-Latex[width=4pt, length=4pt]}]   (0.0,0.75) -- (0.11,0.75);
\draw [arrows = {-Latex[width=4pt, length=4pt]}]   (0.0,0.8) -- (0.075,0.8);
\draw [arrows = {-Latex[width=4pt, length=4pt]}]   (0.0,0.7) -- (0.075,0.7);
\draw (0.75,0.5) node[scale=1] {$\Omega$};
\draw [arrows = {-Latex[width=5pt, length=5pt]}] (0.0,-0.05) -- (1.5,-0.05);
\draw [arrows = {-Latex[width=5pt, length=5pt]}]     (1.5,-0.05) -- (0.0,-0.05);
\draw [arrows = {-Latex[width=4pt, length=4pt]}]   (1.5,0.75) -- (1.6,0.75);
\draw [arrows = {-Latex[width=4pt, length=4pt]}]   (1.5,0.8) -- (1.6,0.8);
\draw [arrows = {-Latex[width=4pt, length=4pt]}]   (1.5,0.7) -- (1.6,0.7);
\draw [arrows = {-Latex[width=4pt, length=4pt]}]   (1.5,0.25) -- (1.6,0.25);
\draw [arrows = {-Latex[width=4pt, length=4pt]}]   (1.5,0.3) -- (1.6,0.3);
\draw [arrows = {-Latex[width=4pt, length=4pt]}]   (1.5,0.2) -- (1.6,0.2);
\draw [arrows = {-Latex[width=5pt, length=5pt]}] (-0.05,0.0) -- (-0.05,1.0);
\draw [arrows = {-Latex[width=5pt, length=5pt]}]   (-0.05,1.0)--(-0.05,0.0);
\draw (0.75,-0.1) node[scale=0.7] {$1.5L$};
\draw (-0.1,0.5) node[scale=0.7] {$1L$};
\draw [arrows = {-Latex[width=5pt, length=5pt]}]   (1.5-0.05,0.5-0.15)--(1.5-0.05,0.5+0.15);
\draw [arrows = {-Latex[width=5pt, length=5pt]}]   (1.5-0.05,0.5+0.15) -- (1.5-0.05,0.5-0.15);
\draw [arrows = {-Latex[width=5pt, length=5pt]}]   (1.5-0.05,0.15) -- (1.5-0.05,0.5-0.15);
\draw [arrows = {-Latex[width=5pt, length=5pt]}]   (1.5-0.05,0.5-0.15) -- (1.5-0.05,0.15);
\draw [arrows = {-Latex[width=5pt, length=5pt]}]   (1.5-0.05,0.5+0.15) -- (1.5-0.05,0.85);
\draw [arrows = {-Latex[width=5pt, length=5pt]}]   (1.5-0.05,0.85) -- (1.5-0.05,0.5+0.15) ;
\draw (1.5-0.15,0.5) node[scale=0.7] {$0.3L$};
\draw (1.5-0.15,0.75) node[scale=0.7] {$0.2L$};
\draw (1.5-0.15,0.25) node[scale=0.7] {$0.2L$};
\end{tikzpicture}
\caption{Design domains for different Examples in 2d space.}
\label{designdomains2d}
\end{figure}
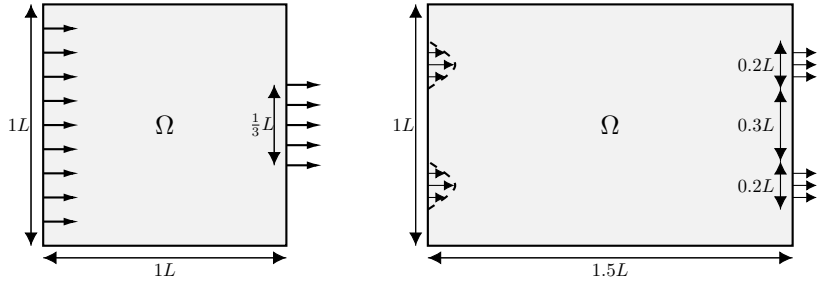

\textbf{Example 1 (Diffuser in 2d)}: Consider a benchmark of the shape design for the diffuser in 2d space (see \cite{Borrvall2003}). The computational domain is set to be a square $\Omega=[0,1]^2$ (see Fig. \ref{designdomains2d} left). The flow on the inlet is imposed by the prescribed function $\vu_d=(1,0)^{\rm T}$. The volume target is set to be $\hat{V}=0.4$ of the solid phase. The basic parameters are given as follows: $\kappa_1=0.001, \kappa_2=0.1, \beta=5$ and $\eta_1 = 1$. Set $N_{\phi} = 10$ for Allen-Cahn flow and $N_{\phi}=1$ for Cahn-Hilliard flow. Choose the initial phase field function be the constant $\phi_0 = 0.5$.
\begin{itemize}
    \item Test the performance of stabilized Allen-Cahn gradient flow by Algorithm \ref{algAC} for realizing the topology optimization. No stabilized terms and projection operator are used for the Allen-Cahn gradient flow in situation 1. Set $S_0=100$ for situation 2, $S_0=100, S_1=1$ for situation 3, and $S_0=100, S_1=1$ as well as the projection operator for situation 4. The maximum and minimum of the phase field function are presented in Table \ref{TabACTest} during shape evolution showing that the stabilized Allen-Cahn gradient flow with the projection operator behaves well even in large time step $\tau = 0.2$ (see optimal shape in Fig. \ref{exp1CHGigTAC} right). Then fix $\tau =0.005, S_0=1$, and $S_1=0.1$. 
    The optimal shapes and the corresponding velocity fields are shown in Fig. \ref{exp1Evo} with different viscous parameters. The curves of convergence histories for total energy and volume errors are presented in Fig. \ref{exp1obj} demonstrating the property of energy decreasing and the high accuracy of volume control.

    \item Next, consider algorithm \ref{algCH} by the Cahn-Hilliard gradient flow for topology optimization. Choose $S_0 = 1$ and $S_1 = 0.5$ for the stabilized parameters. Though the Cahn-Hilliard gradient flow has the property of mass conservation, we note that the mass of solid region $\int_{\Omega}(1-\phi)$ is not equivalent to volume function $V(\phi)$ in which the phase field function may exceed the range $[0,1]$. The stabilized parameters are well chosen such that the phase field function is well controlled in the range $[0,1]$ during shape evolution (see Table \ref{TabCHTest}) to meet the demand of volume constraint. The optimal shapes are shown in Fig. \ref{exp1CHGigTAC} (left and middle) with different viscous parameters. The convergence histories of total energy and mass variation are shown in Fig. \ref{exp1CHobj} showing that the Chan-Hiliard gradient flow scheme proposed has the energy dissipation property.
\end{itemize}

\begin{table}[htbp]
\centering
\begin{tabular}{|c | c c | c  c | c  c   | c  c |}
\hline \ & \multicolumn{2}{c}{Situation 1} & \multicolumn{2}{c}{Situation 2} & \multicolumn{2}{c}{Situation 3} & \multicolumn{2}{c|}{Situation 4}\\
\hline iter & $\max \phi$ & $\min \phi$& $\max \phi$ & $\min \phi$ & $\max \phi$ & $\min \phi$& $\max \phi$ & $\min \phi$ \\
\hline 0   & 0.5    & 0.5     & 0.5  & 0.5  & 0.5 & 0.5    & 0.5  & 0.5 \\
\hline 1   & 1e73   & -1e76   & 1.22 & 0.51 & 0.86 & 0.54  & 0.89 & 0.54 \\
\hline 20  & -      & -       & 1.46 & 0.24 & 1.37 & -0.06 & 1.00 & 0.00\\
\hline 40  & -      & -       & 1.38 &-0.01 & 1.36 & -0.03 & 1.00 & 0.00\\
\hline 60  & -      & -       & 1.35 & 0.00 & 1.34 & -0.02 & 1.00 & 0.00\\
\hline 80  & -      & -       & 1.39 & 0.00 & 1.34 & -0.01 & 1.00 & 0.00\\
\hline 100 & -      & -       & 1.40 & 0.00 & 1.33 & -0.01 & 1.00 & 0.00\\
\hline
\end{tabular}
\caption{Tests for effects of the stabilizers and the projection operator for Example 1 with $\tau =0.2$.}
\label{TabACTest}
\end{table}
\begin{table}[htbp]
\centering
\begin{tabular}{| c c  c  c  c | c  c  c  c  c |}
\hline \multicolumn{5}{|c|}{$\mu=0.01$} & \multicolumn{5}{|c|}{$\mu=0.1$}\\
\hline \ & \multicolumn{2}{c|}{$\tau = 0.0025$} & \multicolumn{2}{c|}{$\tau = 0.01$}& \ & \multicolumn{2}{c|}{$\tau = 0.001$} & \multicolumn{2}{c|}{$\tau = 0.01$}\\
\hline
iter & $\max \phi$ & $\min \phi$& $\max \phi$ & $\min \phi$ & iter & $\max \phi$ & $\min \phi$ & $\max \phi$ & $\min \phi$ \\ \hline 
  0 & 0.65& 0.65  & 0.65 & 0.65  & 0  & 0.65 & 0.65   & 0.65 & 0.65\\
 20 & 0.67& 0.63  & 0.78 & 0.55  & 20 & 0.67 & 0.63   & 0.77 & 0.55\\
 40 & 1.03& -0.04 & 1.02 & -0.02 & 40 & 1.03 & -0.02  & 1.02 & -0.02\\
60  & 1.00& -0.06 & 1.03 & -0.03 & 60 & 1.00 & -0.06  & 1.03 & -0.03\\
80  & 1.01& -0.05 & 1.03 & -0.03 & 80 & 1.01 & -0.05  & 1.03 & -0.03\\
100 & 1.01& -0.05 & 1.03 & -0.03 &100 & 1.01 & -0.05  & 1.03 & -0.03\\
\hline
\end{tabular}
\caption{Tests for range of phase field function by Algorithm \ref{algCH} for Example 1 with different $\tau$ and $\mu$.}
\label{TabCHTest}
\end{table}


\begin{figure}[htbp]
\begin{minipage}[b]{0.33\textwidth}
\centering
    \includegraphics[width=1.8 in]{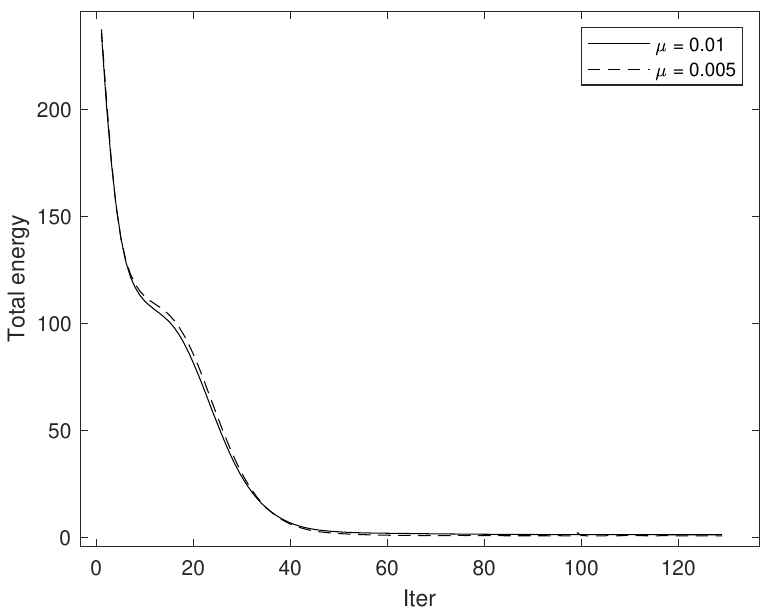}
\end{minipage}
\begin{minipage}[b]{0.33\textwidth}
\centering
    \includegraphics[width=1.8 in]{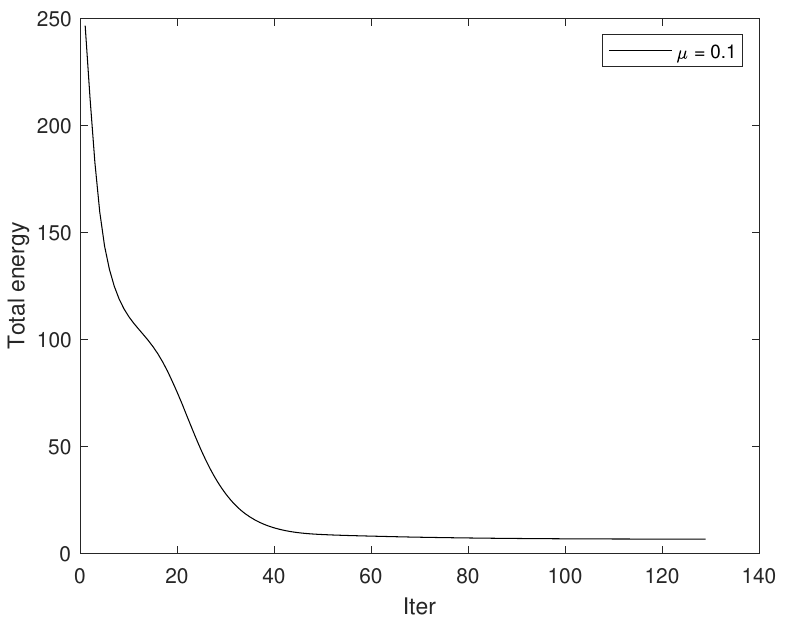}
\end{minipage}
\begin{minipage}[b]{0.33\textwidth}
\centering
    \includegraphics[width=1.8 in]{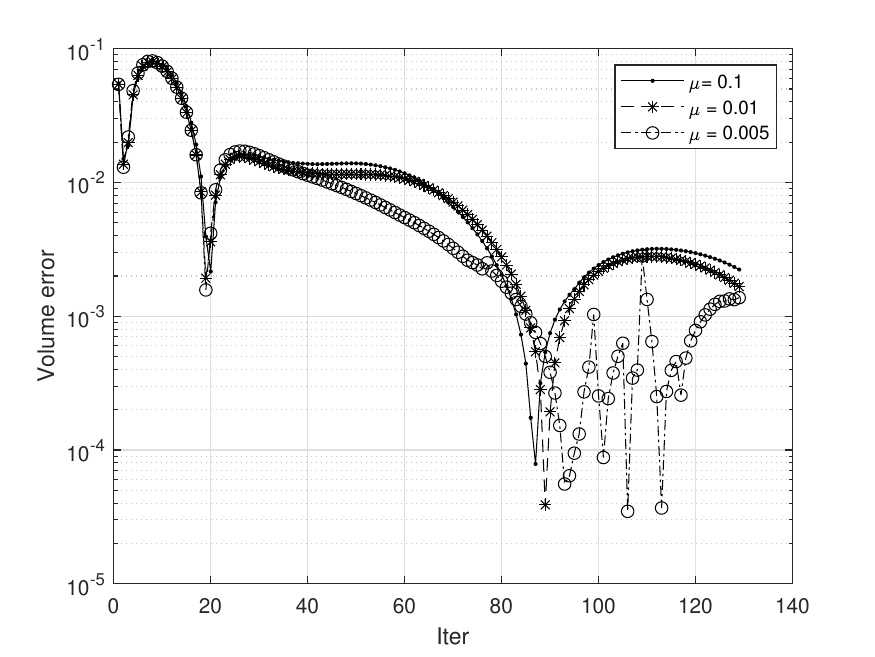}
\end{minipage}
\caption{Convergence histories of energy by Algorithm 1 ($\mu=0.01, 0.005$ left and $0.1$ middle) and volume errors (right) for Example 1.}
\label{exp1obj} 
\end{figure}

\begin{figure}[htbp]
\begin{minipage}[b]{0.33\textwidth}
\centering
    \includegraphics[width=1.67 in]{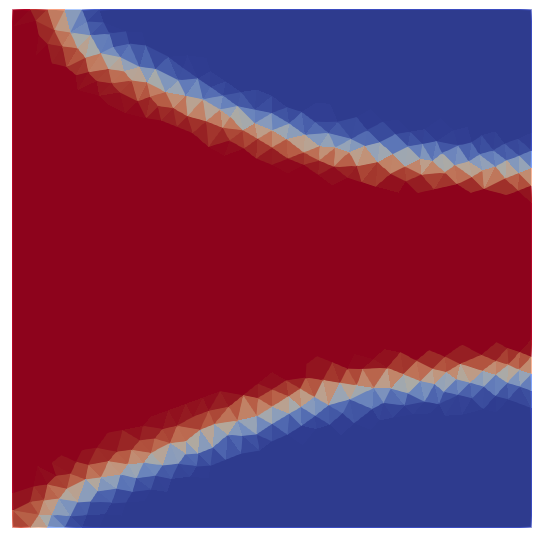}
\end{minipage}
\begin{minipage}[b]{0.33\textwidth}
\centering
    \includegraphics[width=1.65 in]{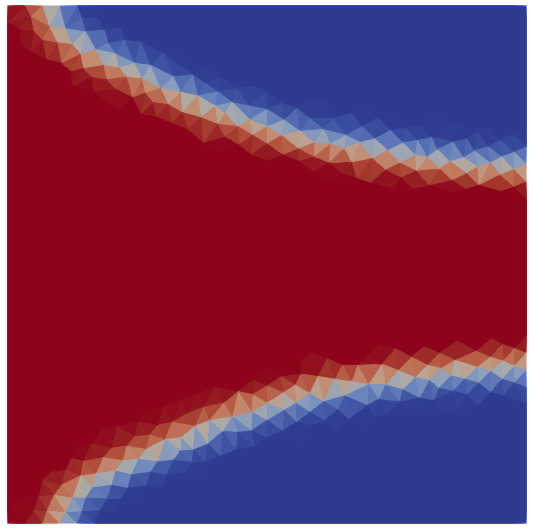}
\end{minipage}
\begin{minipage}[b]{0.33\textwidth}
\centering
    \includegraphics[width=1.7 in]{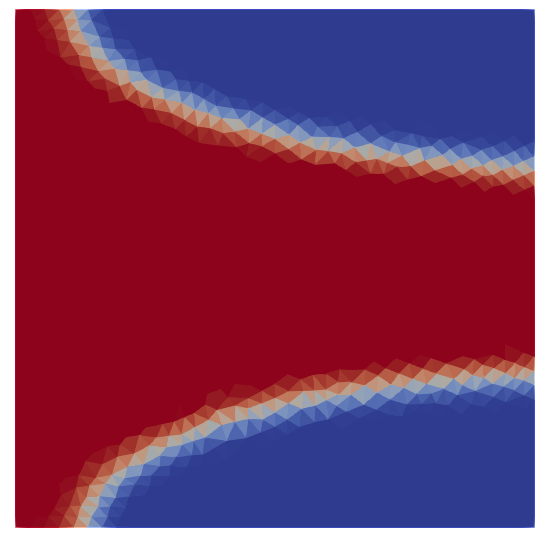}
\end{minipage}
\\
\begin{minipage}[b]{0.33\textwidth}
\centering
    \includegraphics[width=1.9 in]{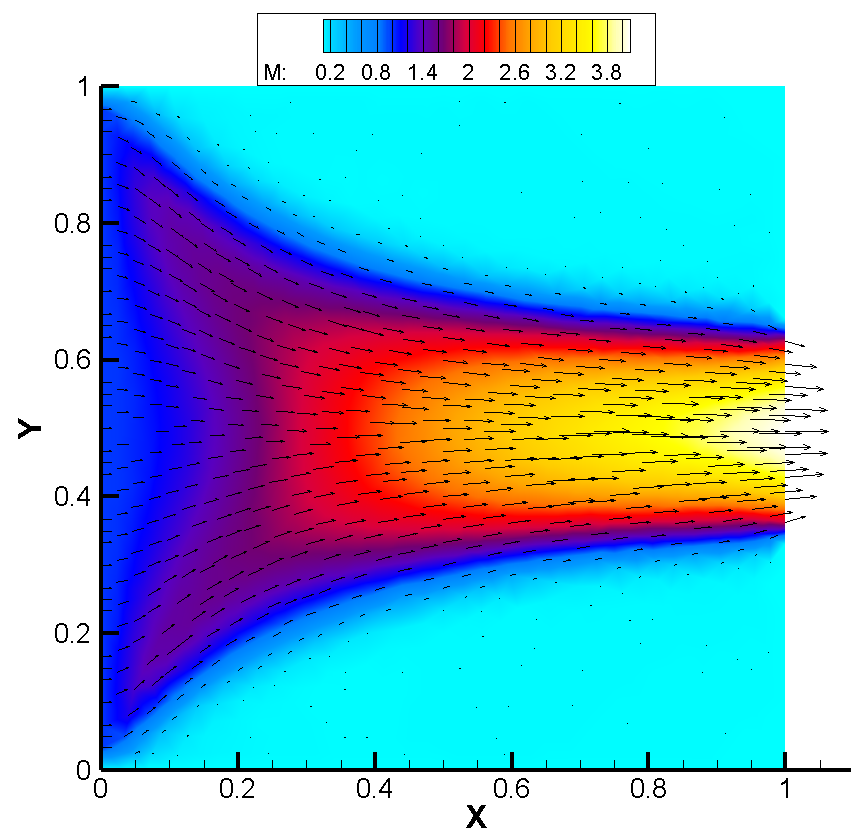}
\end{minipage}
\begin{minipage}[b]{0.33\textwidth}
\centering
    \includegraphics[width=1.9 in]{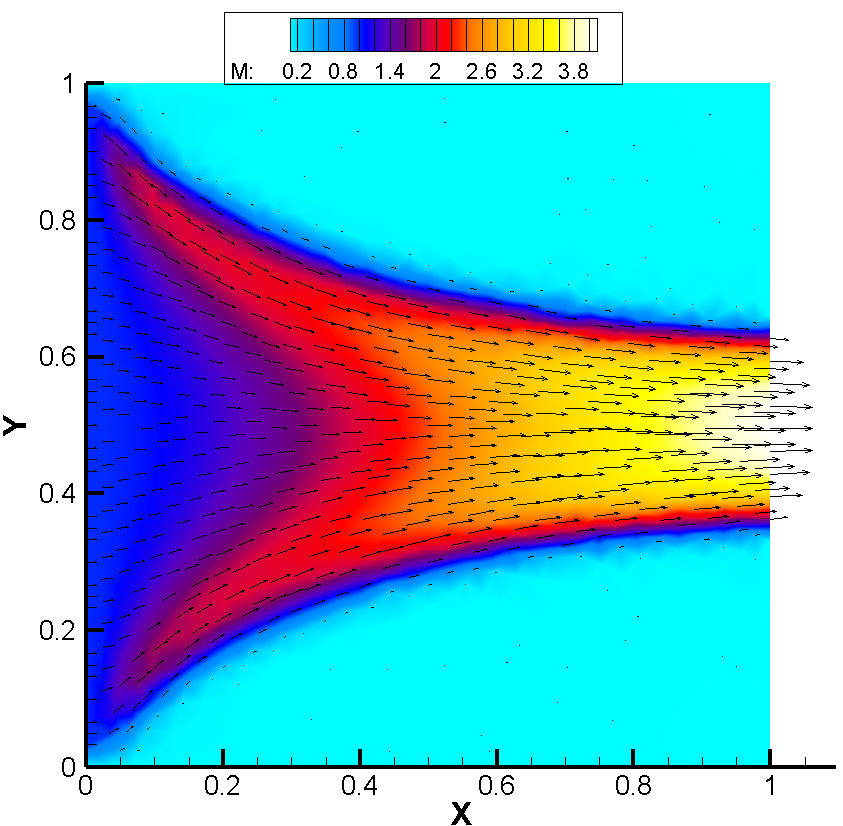}
\end{minipage}
\begin{minipage}[b]{0.33\textwidth}
\centering
    \includegraphics[width=1.9 in]{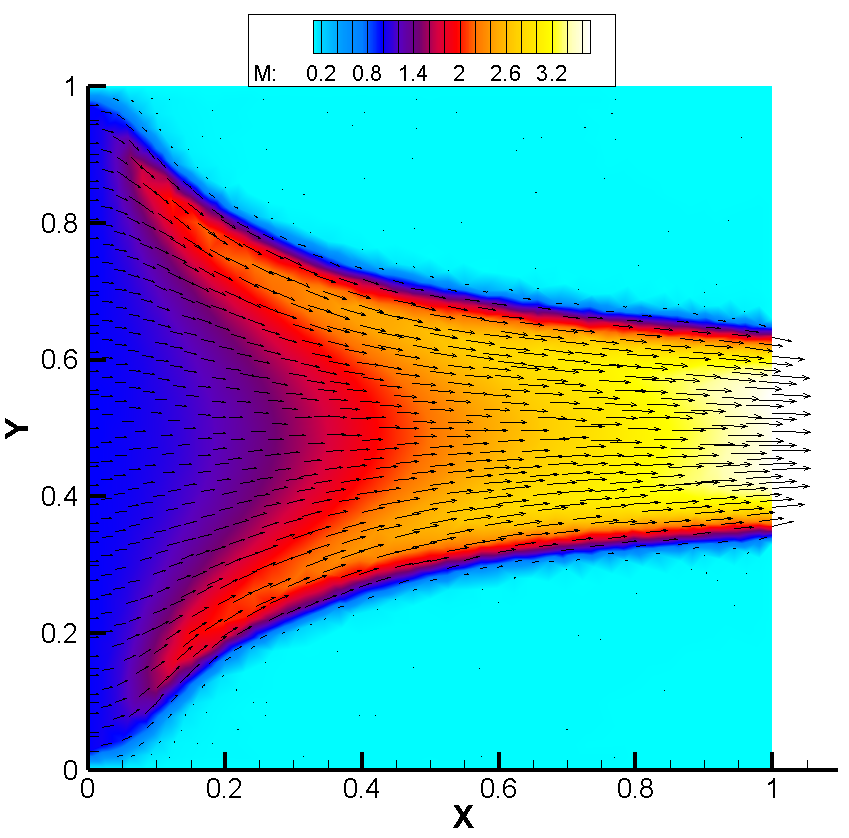}
\end{minipage}
\caption{Optimized distributions (line 1) and velocity fields (line 2) for Example 1: $\mu=0.1$, $0.01$, and $0.005$ from left to right.}
\label{exp1Evo} 
\end{figure}

\begin{figure}[htbp]
\begin{minipage}[b]{0.33\textwidth}
\centering
    \includegraphics[width=1.6 in]{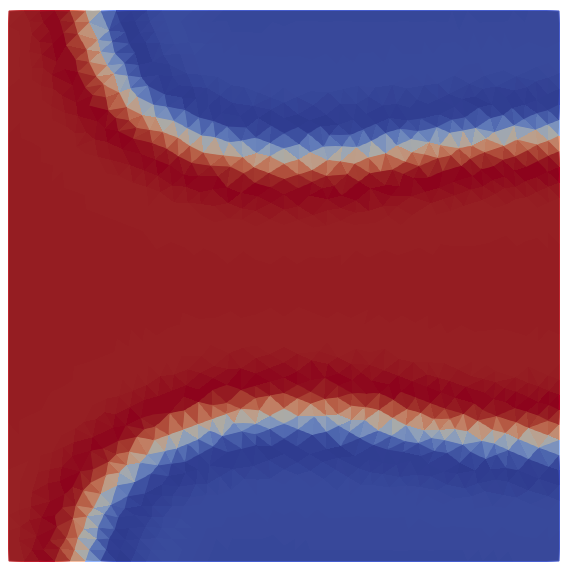}
\end{minipage}
\begin{minipage}[b]{0.33\textwidth}
\centering
    \includegraphics[width=1.6 in]{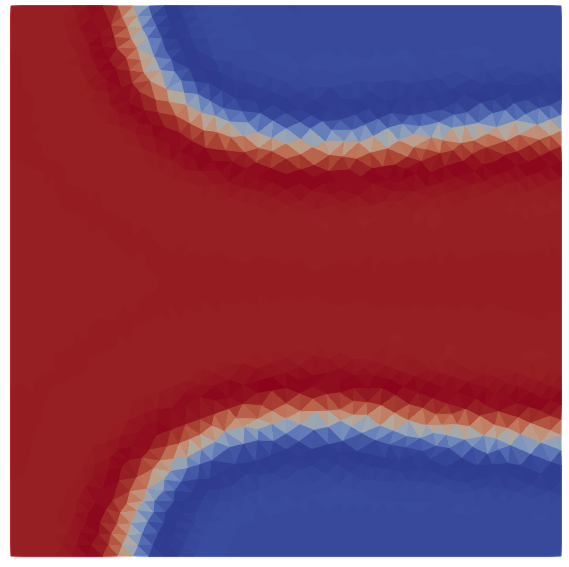}
\end{minipage}
\begin{minipage}[b]{0.33\textwidth}
\centering
    \includegraphics[width=1.6 in]{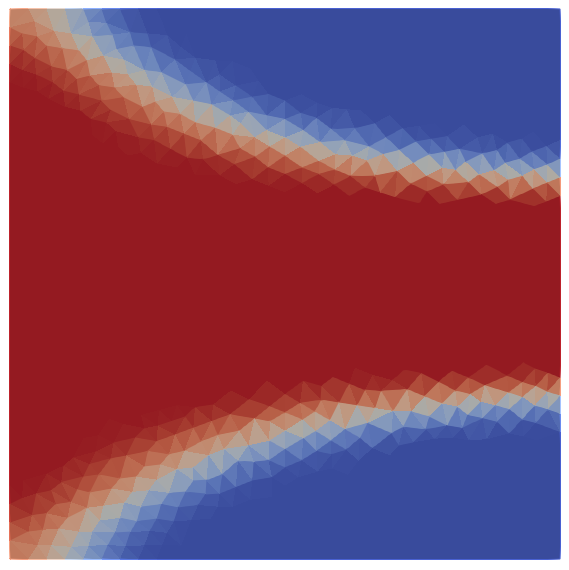}
\end{minipage}
\caption{Optimal distributions by Algorithm \ref{algCH} left ($\mu=0.1$), middle ($\mu=0.01$) and  $\tau =0.2$ right.}
\label{exp1CHGigTAC} 
\end{figure}

\begin{figure}[htbp]
\begin{minipage}[b]{0.5\textwidth}
\centering
    \includegraphics[width=2.0 in]{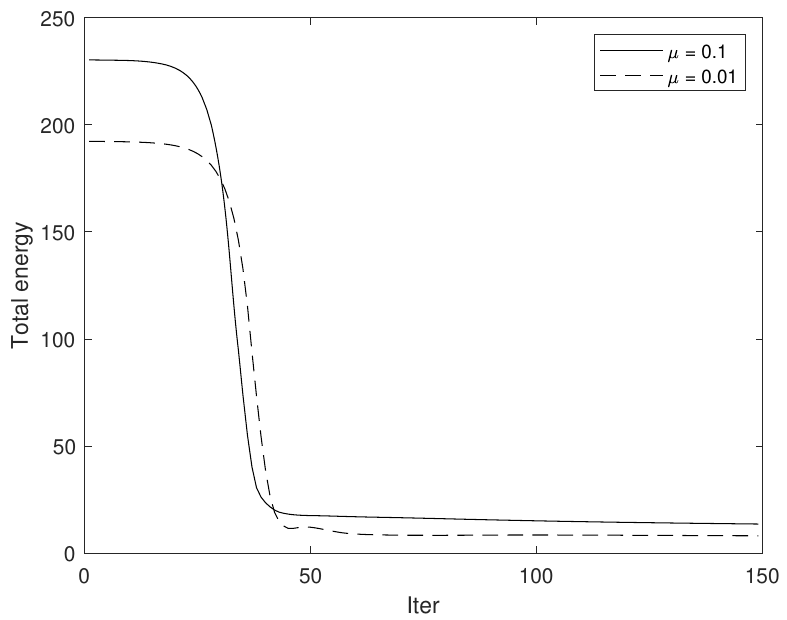}
\end{minipage}
\begin{minipage}[b]{0.5\textwidth}
\centering
    \includegraphics[width=2.0 in]{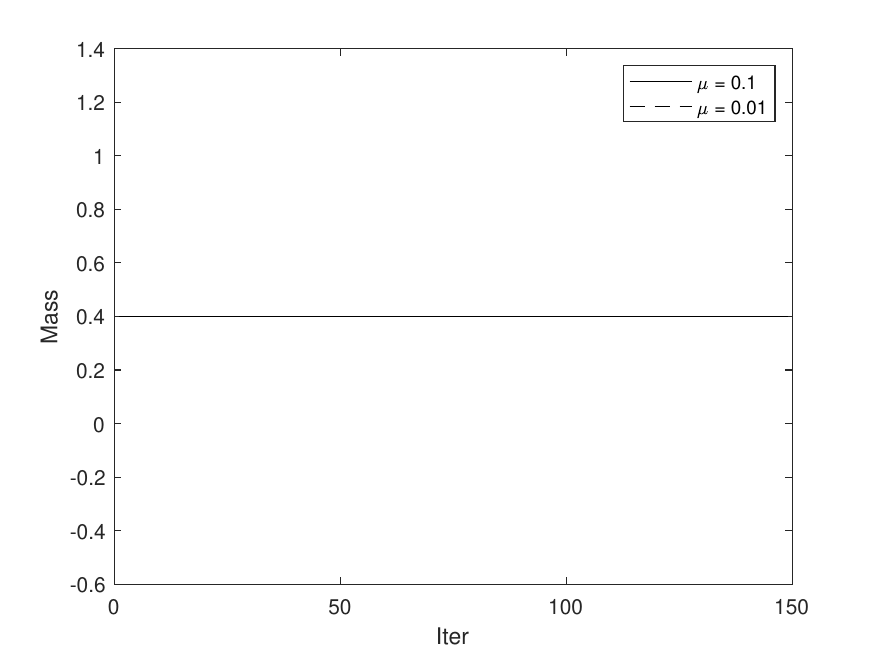}
\end{minipage}
\caption{Convergence histories of energy (left $\mu=0.1$, $\mu=0.01$) and mass (right) for Example 1 by Algorithm \ref{algCH}.}
\label{exp1CHobj} 
\end{figure}

\textbf{Example 2 (Bypass design)}: The second Example is the shape design for the bypass in 2d space. The computational domain is set to be a rectangle $\Omega=[0,1.5]\times[-0.5,0.5]$ (see Fig. \ref{designdomains2d} right) with two inlets and two outlets. The flow on the inlet is imposed by the prescribed function $\vu_d=(-50(y^2-0.35^2)(y^2-0.15^2),0)^{\rm T}$. 
\begin{itemize}
    \item At first, fix the volume target $V_0=0.85$ in which the solid phase occupies almost $57\%$ the volume of the whole domain. Consider the algorithm \ref{algAC} with the projected Allen-Cahn gradient flow scheme to evolve the phase field function. The basic parameters are given as follows: $\tau = 0.0005, \kappa_1=0.001, \kappa_2=0.1, \beta=500, S_0=1, S_1=0.5, \eta_1=90$ and $N_{\phi} = 10$. The initial phase field function is $\phi_0=\min({\rm abs}(y-0.3)-0.1,{\rm abs}(y+0.3)-0.1)$. The optimal distribution with single connected shape and the corresponding velocity fields are presented on the left and middle of Fig. \ref{exp2Evo} where the viscous coefficients $\mu = 0.1$ and $\mu = 0.01$. The convergence histories of total energy are shown in Fig. \ref{exp2obj} left demonstrating that the algorithm \ref{algAC} has the energy dissipative property. Meanwhile, the volume error (Fig. \ref{exp2obj} right) is well controlled almost by $0.01$.

    \item  Consider the  Cahn-Hilliard gradient flow algorithm \ref{algCH} for topology optimization. Choose the initial phase field function $\phi_0 = 0.5$. The basic parameters are given as follows: $\mu=0.01, \tau=0.00025, \kappa_1=0.001, \kappa_2=0.01, S_0=1.0, S_1=0.15$ and $\eta_1=4$. The other conditions are the same as the above. The optimal distribution and its corresponding velocity field exhibit the double channels in Fig. \ref{exp2Evo} right. The convergence histories of total energy is presented in the middle of Fig. \ref{exp2obj} showing the effectiveness of the algorithm \ref{algCH}.  Furthermore, different $\mu$ (see Fig. \ref{exp2CHEvo}) have been tested to display the optimal configurations and curves of total free energy.
\end{itemize}

\begin{figure}[htbp]
\begin{minipage}[b]{0.33\textwidth}
\centering
    \includegraphics[width=1.8 in]{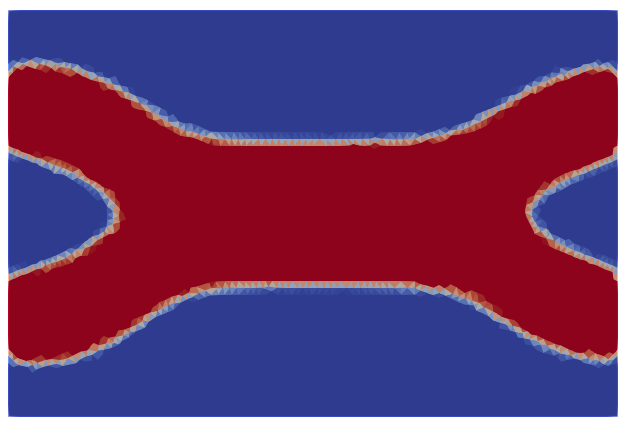}
\end{minipage}
\begin{minipage}[b]{0.33\textwidth}
\centering
    \includegraphics[width=1.8 in]{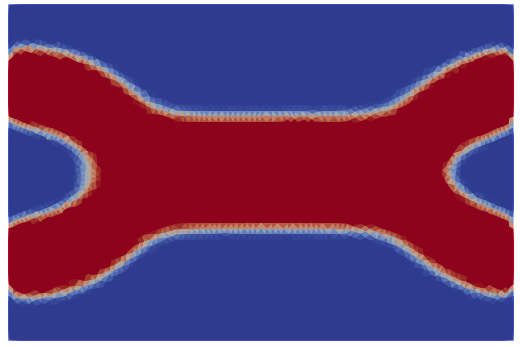}
\end{minipage}
\begin{minipage}[b]{0.33\textwidth}
\centering
    \includegraphics[width=1.8 in]{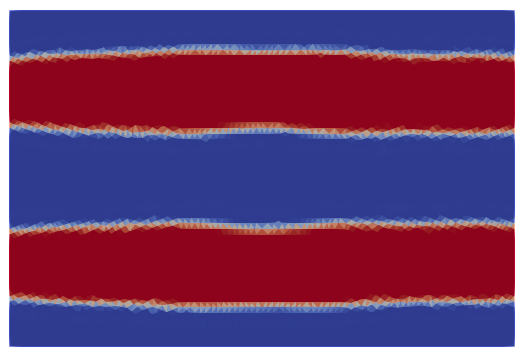}
\end{minipage}
\\
\begin{minipage}[b]{0.33\textwidth}
\centering
    \includegraphics[width=1.8 in]{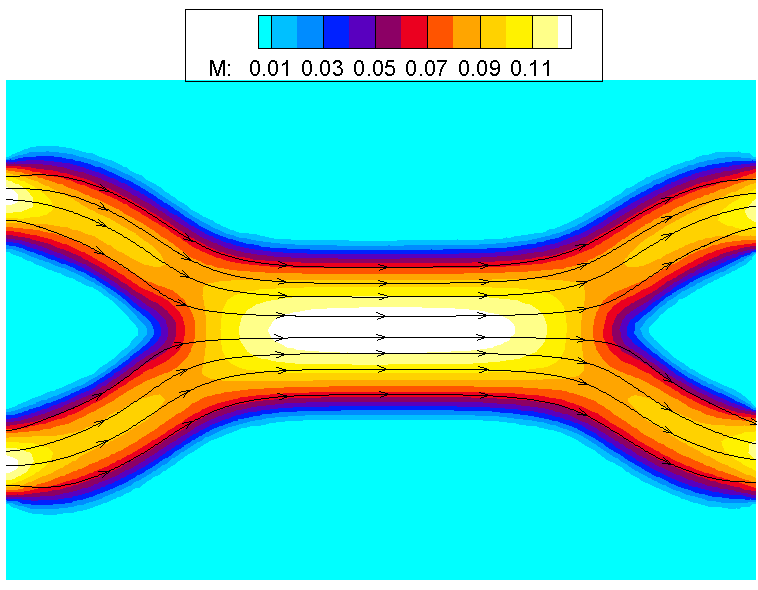}
\end{minipage}
\begin{minipage}[b]{0.33\textwidth}
\centering
    \includegraphics[width=1.8 in]{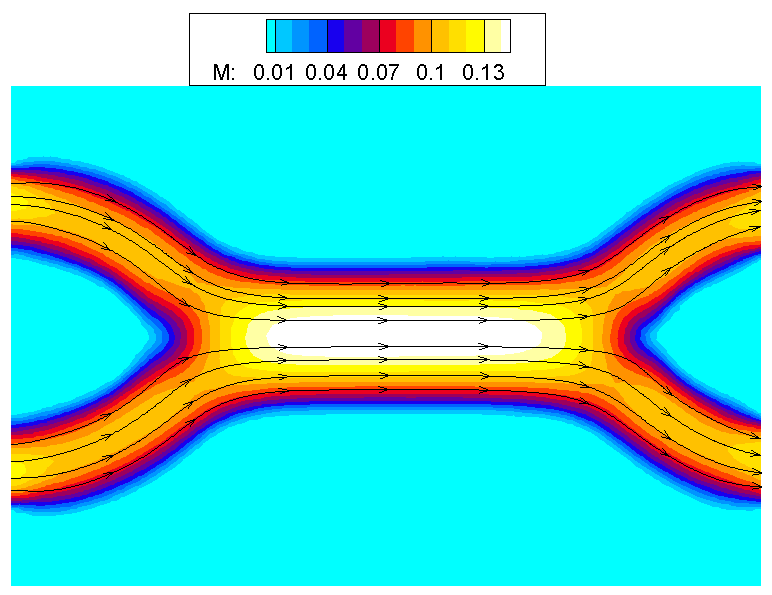}
\end{minipage}
\begin{minipage}[b]{0.33\textwidth}
\centering
    \includegraphics[width=1.8 in]{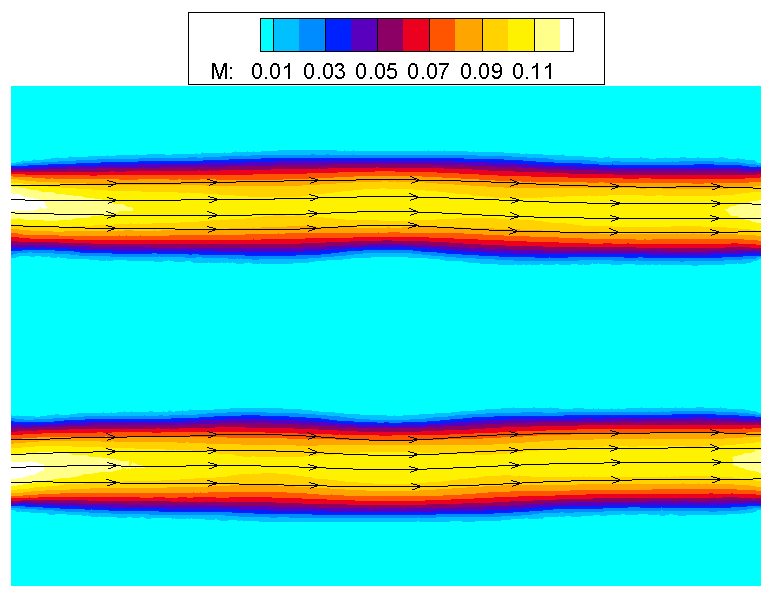}
\end{minipage}
\caption{Optimal distributions on line 1 and velocity fields on line 2 (Allen-Cahn: $\mu=0.1$ left, $\mu=0.01$ middle and Cahn-Hilliard: $\mu=0.01$ right) for Example 2.}
\label{exp2Evo} 
\end{figure}

\begin{figure}[htbp]
\begin{minipage}[b]{0.33\textwidth}
\centering
    \includegraphics[width=1.6 in]{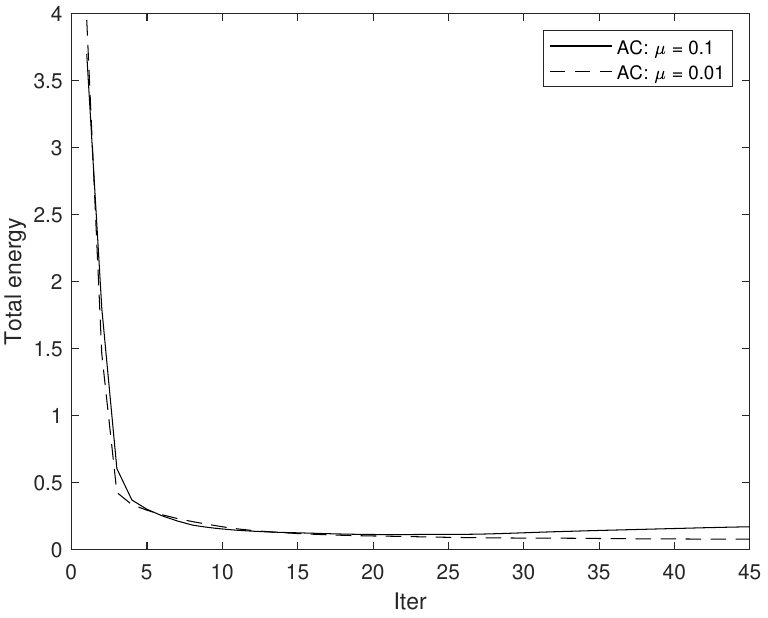}
\end{minipage}
\begin{minipage}[b]{0.33\textwidth}
\centering
    \includegraphics[width=1.6 in]{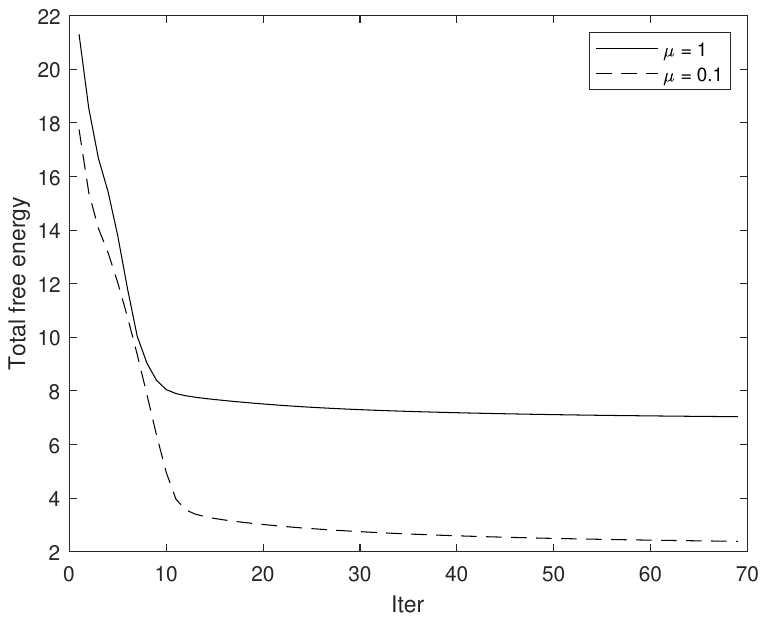}
\end{minipage}
\begin{minipage}[b]{0.33\textwidth}
\centering
    \includegraphics[width=1.6 in]{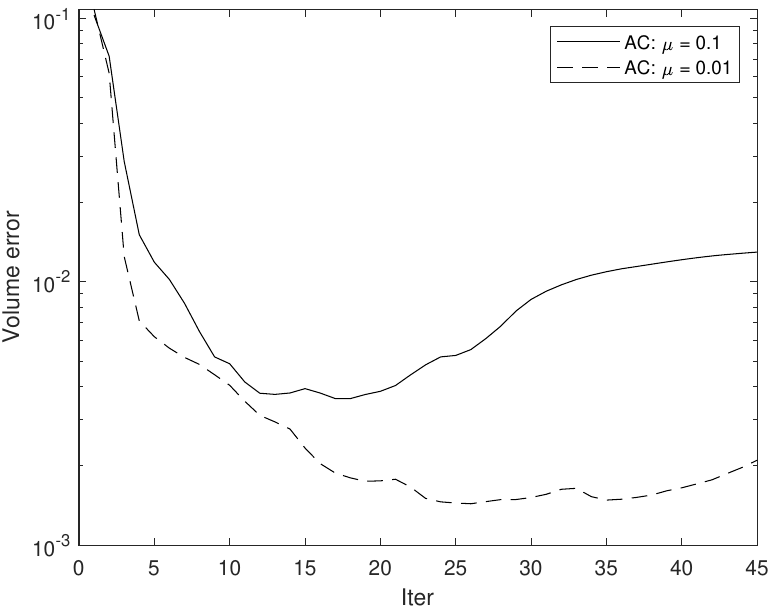}
\end{minipage}
\caption{The convergence histories of total energy (Allen-Cahn: left, Chan-Hilliard: middle) and volume errors for Allen-Cahn (right) for Example 2.}
\label{exp2obj} 
\end{figure}

\begin{figure}[htb]
\begin{minipage}[b]{0.33\textwidth}
\centering
    \includegraphics[width=1.7 in]{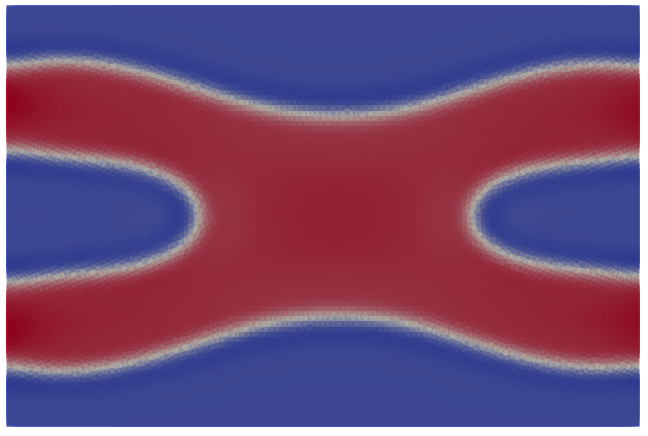}
\end{minipage}
\begin{minipage}[b]{0.33\textwidth}
\centering
    \includegraphics[width=1.7 in]{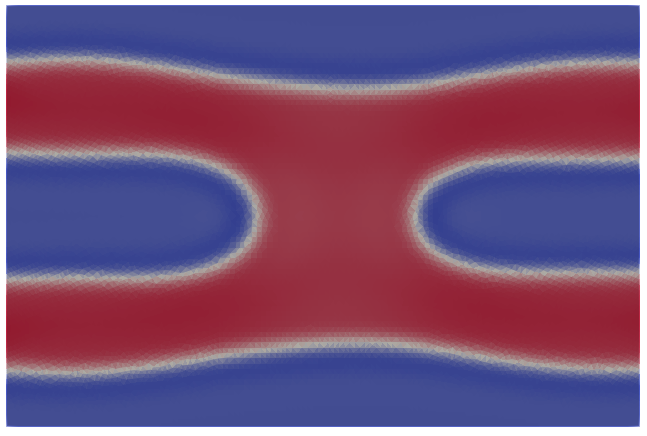}
\end{minipage}
\begin{minipage}[b]{0.33\textwidth}
\centering
    \includegraphics[width=1.5 in]{Exp2CHObj.pdf}
\end{minipage}
\caption{Optimal distributions (Cahn-Hilliard: $\mu=1$ left, $\mu=0.1$ middle and convergence histories right) for Example 2.}
\label{exp2CHEvo} 
\end{figure}

\textbf{Example 3 (Bypass design in 3d)}: This example is solved by Algorithm \ref{algAC}. See Fig. \ref{DesignDomain3d} for the domains to design the internal flow channels.
\begin{itemize}
    \item (A). The design domain is the Fig. \ref{DesignDomain3d} left. The velocity on the inlet is $\vu_d = (0,0,-1)^{\rm T}$. Set the volume target 0.85 for the solid phase. The basic parameters are given as follows: $\tau =0.0025, \epsilon_1 = 0.0001, s_0=1, \epsilon_2 = 0.1, \beta = 10, \eta_1=10$ and $N_{\phi}$=10. The optimal distribution from two directions and corresponding velocity fields are displayed in Fig. \ref{Exp3AOptPhiVel} with different viscous coefficients. The curves of convergence histories for both total energy and volume in Fig. \ref{objExp3A} shows the energy dissipative for the scheme in 3d space.

    \item  (B). The design domain is the Fig. \ref{DesignDomain3d} right. The velocity on the inlet is $\vu_d = [-(50(y-0.35)(y-0.15)(y+0.35)(y+0.15)),0,0]^{\rm T}$. The basic parameters are given as follows: $\tau =0.001, \epsilon_1 = 0.001, \epsilon_2 = 0.1, \beta = 250, \eta_1=10, V_0=0.55|\Omega|$, $N_{\phi}=15$, $S_0=1$ and $S_1=0.5$. The optimal distribution from two directions and corresponding velocity fields are displayed in Fig. \ref{Exp3BOptPhiVel} with different viscous coefficients. The curves of convergence histories for both total energy and volume in Fig. \ref{objExp3B} shows the energy dissipative for the scheme in 3d space.
\end{itemize}
\begin{figure}[htbp]
\centering
\begin{tikzpicture}[scale=2.5]
\draw[thick] (0.0,0.0) -- (1.0,0.0);
\draw[thick] (0.0,0.0) -- (0.0,1.0);

\draw[thick] (0.0,1.0) -- (1.0,1.0);
\draw[thick] (0.0,1.0) -- (0.353,1.353);
\draw[thick] (1.0,1.0) -- (1.353,1.353);
\draw[thick] (0.353,1.353) -- (1.353,1.353);
\draw[thick] (1.353,1.353) -- (1.353,0.353);
\draw[thick] (1.0,0.0) -- (1.353,0.353);
\draw[dashed,thick] (0.353,1.353) -- (0.353,0.353);
\draw[dashed,thick] (0.0,0.0) -- (0.353,0.353);
\draw[dashed,thick] (1.353,0.353) -- (0.353,0.353);


\draw[thick, fill=gray!30!white] (0.353/3+0.45,1.353/3+0.8) -- (1.353/3+0.45,1.353/3+0.8) -- (1.0/3+0.45,1.0/3+0.8) -- (0.0+0.45,1.0/3+0.8) -- cycle;

\draw[thick,dashed, fill=gray!30!white] (0.353/4+0.2,1.353/4-0.2) -- (1.353/4+0.2,1.353/4-0.2) -- (1.0/4+0.2,1.0/4-0.2) -- (0.0+0.2,1.0/4-0.2) -- cycle;

\draw[thick,dashed,fill=gray!30!white] (0.353/4+0.75,1.353/4-0.05) -- (1.353/4+0.75,1.353/4-0.05) -- (1.0/4+0.75,1.0/4-0.05) -- (0.0+0.75,1.0/4-0.05) -- cycle;

\draw[thick] (1.0,0.0) -- (1.0,1.0);



\draw (0.65,1.0/3+0.9) node[thick, scale=1.0] {$\Gamma_{\rm in}$};
\draw (0.88,0.25-0.15)  node[thick, scale=1.0] {$\Gamma_{\rm out,2}$};
\draw (0.52,0.25-0.06)  node[thick, scale=1.0] {$\Gamma_{\rm out,1}$};

\end{tikzpicture}
\qquad
\begin{tikzpicture}[scale=2.7]
	\draw[thick, fill=white] (0.0,0.0) -- (1.0,0.0) -- (1.0,0.2) -- (0.0,0.2) -- cycle;
	\draw[dashed, fill=white] (1.0,1.0) -- (2.0,1.0) -- (2.0,1.2) -- (1.0,1.2) -- cycle;
	\draw[thick] (1.0,1.2) -- (2.0,1.2) -- (2.0,1.0);
	\draw[thick] (0.0,0.2) -- (1.0,1.2);
	\draw[dashed] (0.0,0.0) -- (1.0,1.0);
	\draw[thick] (1.0,0.2) -- (2.0,1.2);
	\draw[thick] (1.0,0.0) -- (2.0,1.0);
	\draw[thick, fill=gray!30!white] (0.2,0.05)--(0.4,0.05)--(0.4,0.15)--(0.2,0.15) -- cycle;
	\draw[thick, fill=gray!30!white] (0.6,0.05)--(0.8,0.05)--(0.8,0.15)--(0.6,0.15) -- cycle;
	\draw[dashed, fill=gray!30!white] (1.2,1.05)--(1.4,1.05)--(1.4,1.15)--(1.2,1.15) -- cycle;
	\draw[dashed, fill=gray!30!white] (1.6,1.05)--(1.8,1.05)--(1.8,1.15)--(1.6,1.15) -- cycle;
	\draw (0.5,0.1) node[scale=1.0] {$\Gamma_{\text{in}}$};
	\draw (1.5,1.1) node[scale=1.0] {$\Gamma_{\text{out}}$};
	\end{tikzpicture}
 	\caption{The design domains for different Examples in 3d space.}
  	\label{DesignDomain3d} 
\end{figure}
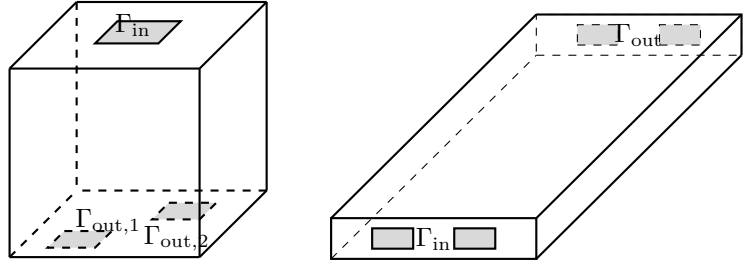
\begin{figure}[htbp]
	\begin{minipage}[b]{0.5\textwidth}
		\centering
		\includegraphics[width=2.1 in]{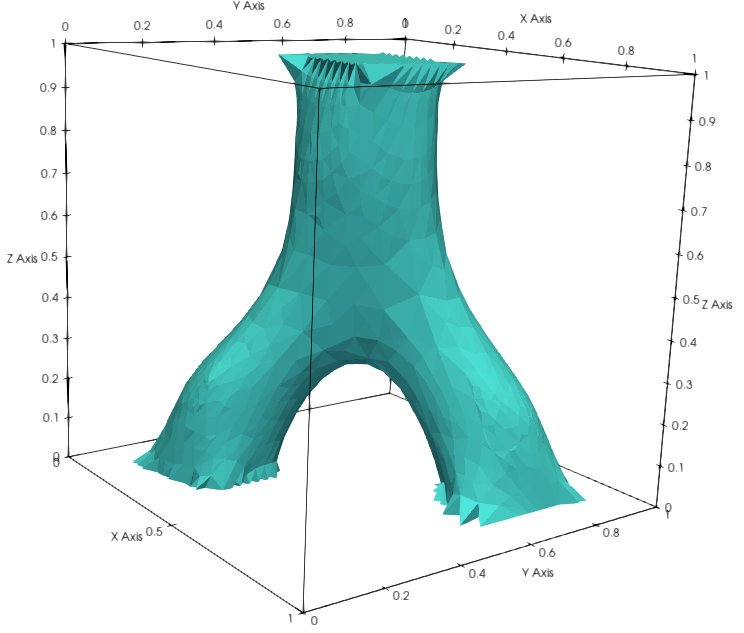}
	\end{minipage}
    \begin{minipage}[b]{0.5\textwidth}
		\centering
		\includegraphics[width=1.8 in]{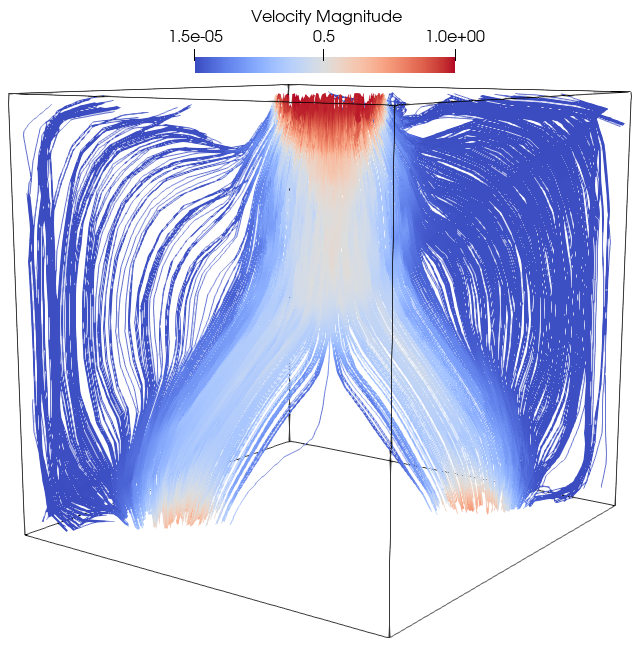}
	\end{minipage}
 \\
 	\begin{minipage}[b]{0.5\textwidth}
		\centering
		\includegraphics[width=2.1 in]{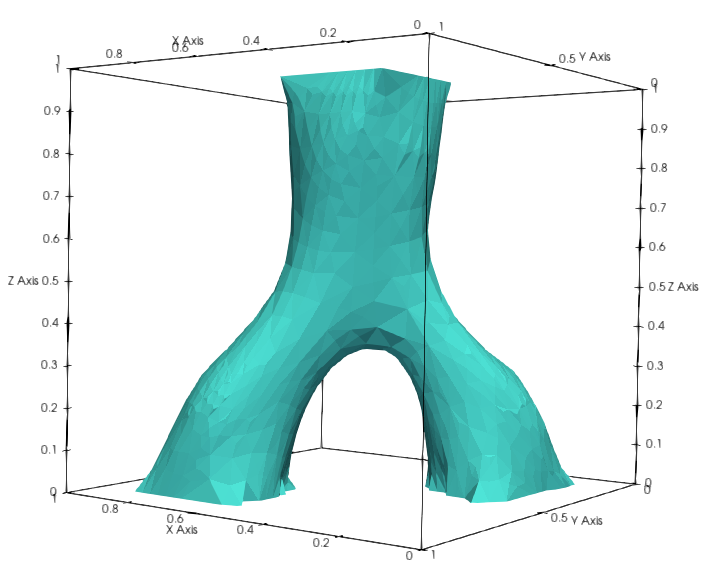}
	\end{minipage}
    \begin{minipage}[b]{0.5\textwidth}
		\centering
		\includegraphics[width=1.8 in]{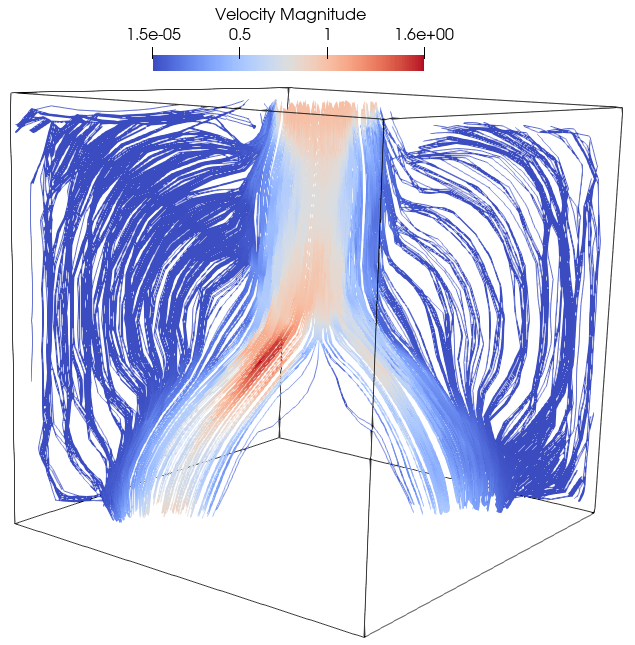}
	\end{minipage}
	\caption{Optimized designs with 36864 tetrahedron elements and corresponding velocity fields of Example 3 (A) (line 1 for $\mu=0.1$ and line 2 for $\mu=0.01$).}
	\label{Exp3AOptPhiVel} 
\end{figure}

 \begin{figure}[htbp]
\begin{minipage}[b]{0.5\textwidth}
\centering
    \includegraphics[width=2.0 in]{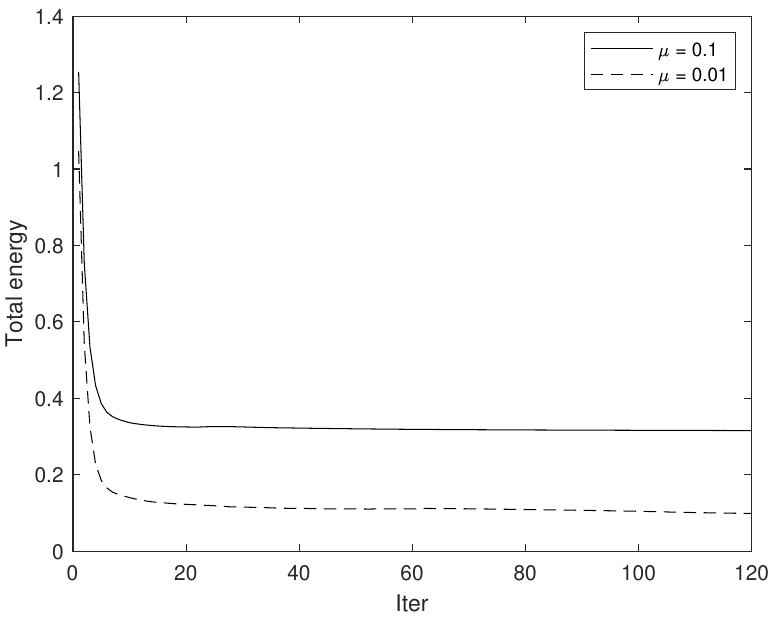}
\end{minipage}
\begin{minipage}[b]{0.5\textwidth}
\centering
    \includegraphics[width=2.0 in]{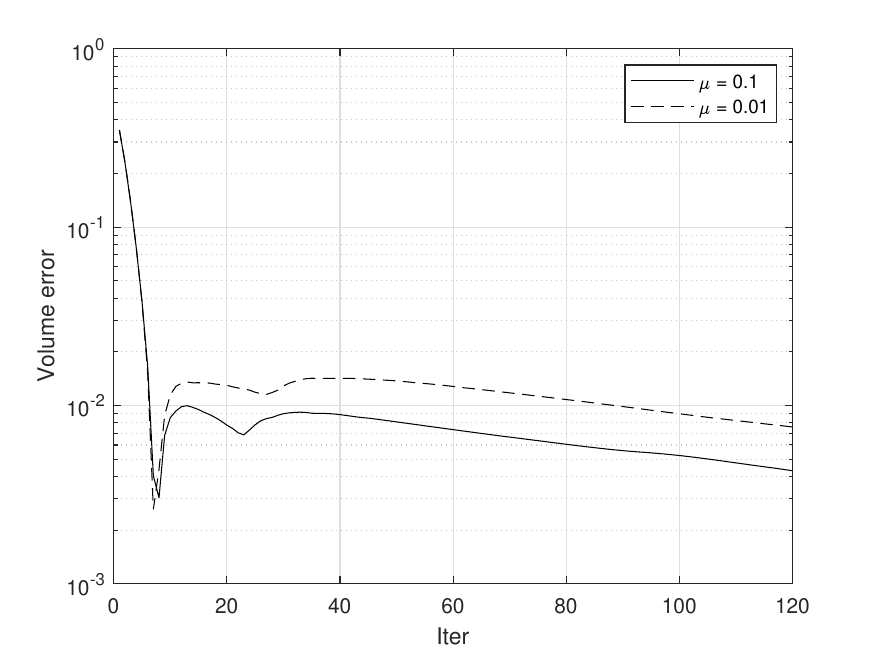}
\end{minipage}
\caption{Convergence histories of the total energy and volume error for Example 3 (A).}
\label{objExp3A} 
\end{figure}

\begin{figure}[htbp]
	\begin{minipage}[b]{0.5\textwidth}
		\centering
		\includegraphics[width=2.3 in]{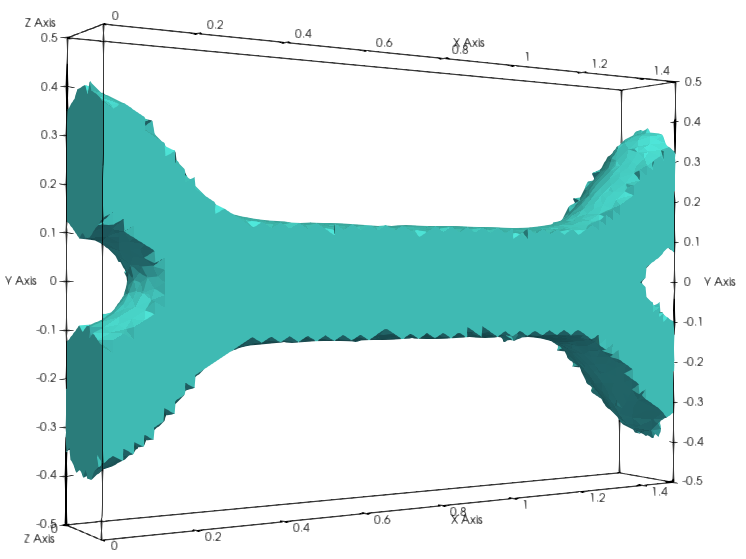}
	\end{minipage}
    \begin{minipage}[b]{0.5\textwidth}
		\centering
		\includegraphics[width=2.2 in]{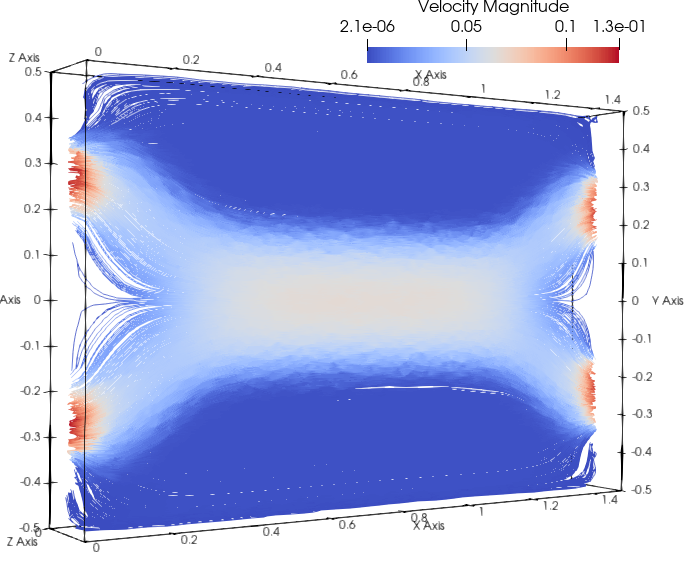}
	\end{minipage}
 \\
	\begin{minipage}[b]{0.5\textwidth}
		\centering
		\includegraphics[width=2.3 in]{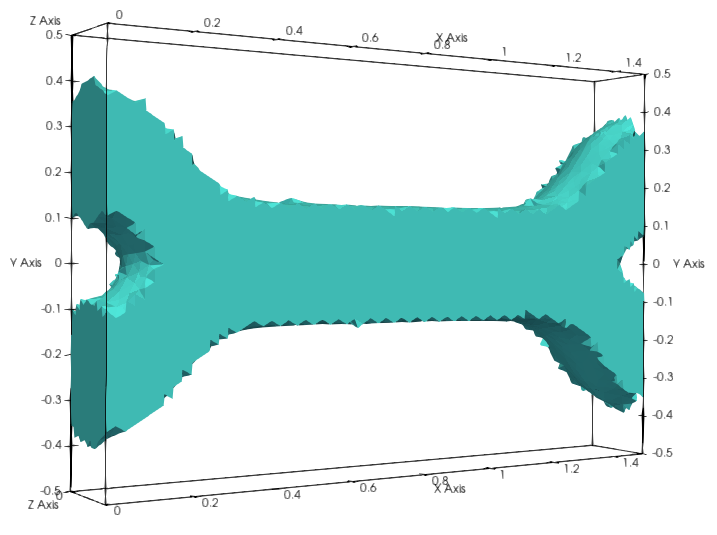}
	\end{minipage}
    \begin{minipage}[b]{0.5\textwidth}
		\centering
		\includegraphics[width=2.2 in]{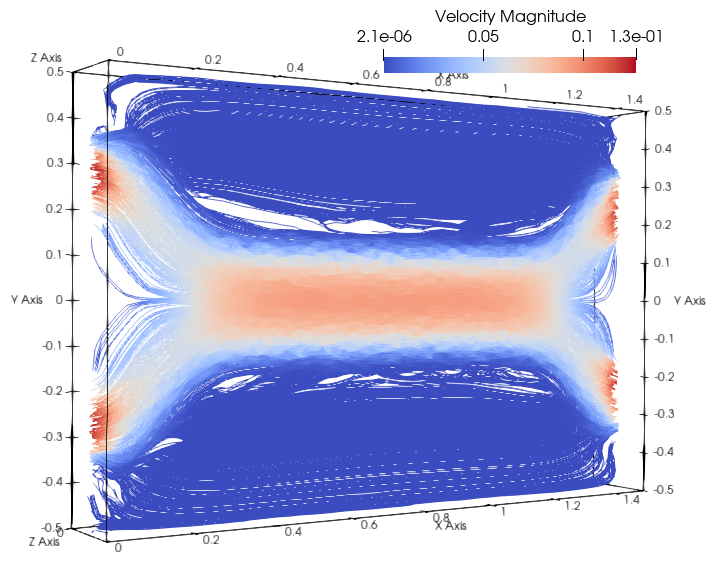}
	\end{minipage}
	\caption{Optimal designs with 89711 tetrahedron elements and corresponding velocity fields of Example 3 (A) (line 1 for $\mu=0.1$ and line 2 for $\mu=0.01$).}
	\label{Exp3BOptPhiVel} 
\end{figure}

\begin{figure}[htbp]
\begin{minipage}[b]{0.33\textwidth}
\centering
    \includegraphics[width=1.8 in]{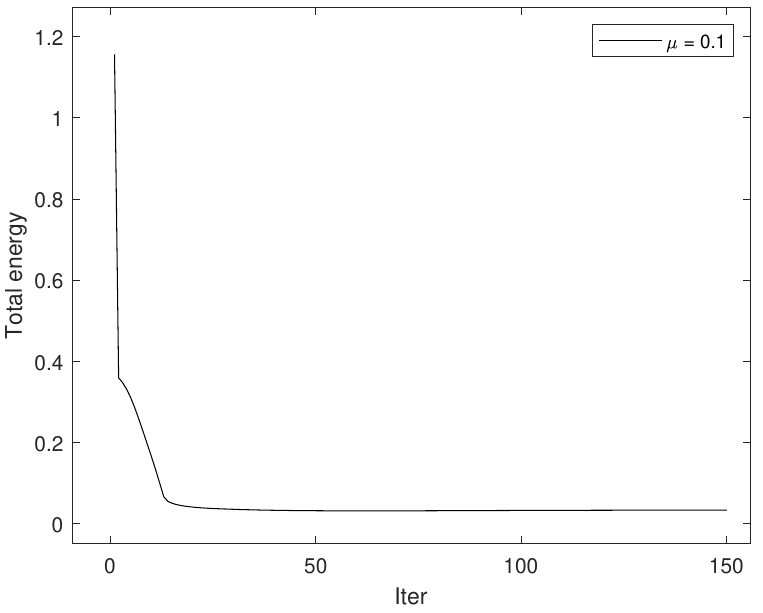}
\end{minipage}
\begin{minipage}[b]{0.33\textwidth}
\centering
    \includegraphics[width=1.8 in]{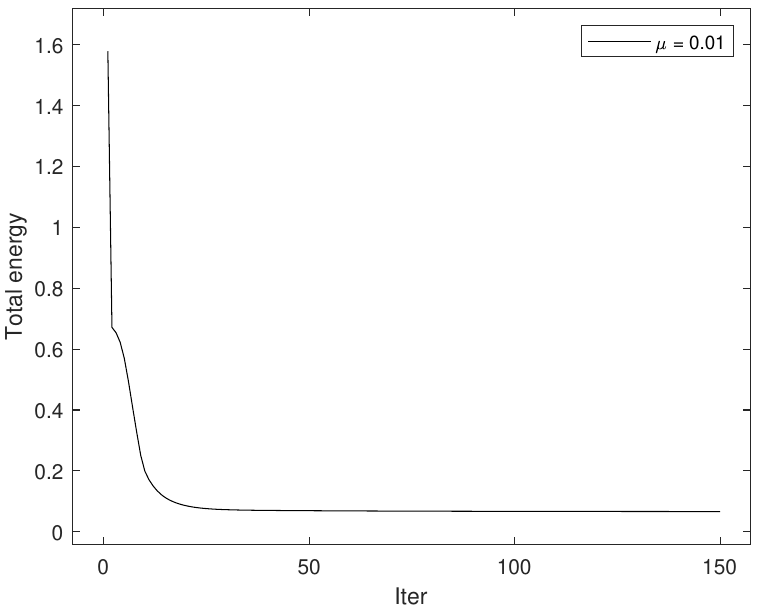}
\end{minipage}
\begin{minipage}[b]{0.33\textwidth}
\centering
    \includegraphics[width=1.8 in]{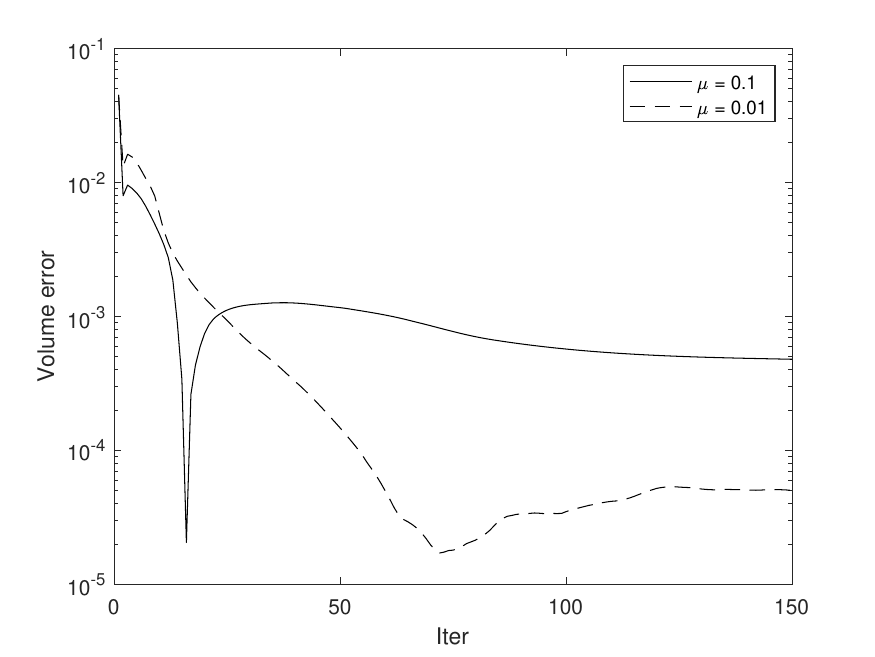}
\end{minipage}
\caption{Convergence histories of the total energy and volume error for Example 3 (B).}
\label{objExp3B} 
\end{figure}
\section{Conclusion}
Topology optimization in incompressible Navier-Stokes equations has been considered using a phase field model. We propose the novel stabilized semi-implicit schemes of the gradient flow in Allen-Cahn and Cahn-Hilliard types for solving the resulting optimal control problem. The unconditional energy stability is shown for the gradient flow schemes in both continuous and discrete spaces by the Lipschtiz continuity. Numerical examples of computational fluid dynamics show effectiveness and robustness of the optimization algorithm proposed. The stabilized gradient flow scheme is the constant coefficient equation that can be solved efficiently. The scheme keeping energy dissipation is simple to realize and may can be extended to other topology optimization models with general objective functionals such as geometric inverse problems or with nonlinear physical constraints.

\end{document}